\documentclass[11pt]{article}
\usepackage{geometry}                
\usepackage{graphicx}
\usepackage{amssymb}
\usepackage{epstopdf}
\usepackage[T1]{fontenc}
\usepackage[latin1]{inputenc}
\usepackage[ngerman, english]{babel}
\usepackage{mathtools}
\usepackage{bbm}
\usepackage{mathrsfs}
\usepackage{amsmath,amscd}
\usepackage{tikz}
\usepackage{tikz-cd}
\usepackage{babel}
\usetikzlibrary{arrows, matrix}
\usetikzlibrary{cd}
\usepackage[arrow, matrix, curve]{xy}

\newcommand{\qed}{\hfill \ensuremath{\Box}}
\newenvironment{proof}{\vspace{1ex}\noindent{\it Proof.}\hspace{0.5em}}
	{\hfill\qed\vspace{1ex}}

\newtheorem{theorem}{Theorem}[section]
\newtheorem{lemma}[theorem]{Lemma}
\newtheorem{proposition}[theorem]{Proposition}
\newtheorem{corollary}[theorem]{Corollary}

\newtheorem{definition}[theorem]{Definition}
\newtheorem{question}[theorem]{Question}
\DeclareMathOperator{\Gal}{\operatorname{Gal}}
\DeclareMathOperator{\Q}{\mathbf{Q}}

\DeclareMathOperator{\Z}{\mathbf{Z}}

\DeclareMathOperator{\N}{\mathbf{N}}

\DeclareMathOperator{\Spec}{\operatorname{Spec}}

\DeclareMathOperator{\Hom}{\operatorname{Hom}}
\DeclareMathOperator{\Frac}{\operatorname{Frac}}

\DeclareMathOperator{\ord}{\operatorname{ord}}

\DeclareMathOperator{\Lie}{\mathrm{Lie}}

\DeclareMathOperator{\Ext}{\operatorname{Ext}}

\DeclareMathOperator{\Og}{\mathcal{O}}

\DeclareMathOperator{\Pic}{\mathrm{Pic}}

\DeclareMathOperator{\et}{\acute{\mathrm{e}}{\mathrm{t}}}
\DeclareMathOperator{\Gm}{\mathbf{G}_m}

\DeclareMathOperator{\Res}{\mathrm{Res}}

\DeclareMathOperator{\alg}{{^\mathrm{alg}}}
\DeclareMathOperator{\sep}{{^\mathrm{sep}}}
\DeclareMathOperator{\perf}{{^\mathrm{perf}}}

\title{Néron models of pseudo-Abelian varieties}
\author{Otto Overkamp}
\date{}
\begin{document}
\maketitle
{\abstract{We study Néron models of pseudo-Abelian varieties over excellent discrete valuation rings of equal characteristic $p>0$ and generalize the notions of \it good reduction \rm and \it semiabelian reduction \rm to such algebraic groups. We prove that the well-known representation-theoretic criteria for good and semiabelian reduction due to Néron-Ogg-Shafarevich and Grothendieck carry over to the pseudo-Abelian case, and give examples to show that our results are the best possible in most cases. Finally, we study the order of the group scheme of connected components of the Néron model in the pseudo-Abelian case. Our method is able to control the $\ell$-part (for $\ell\not=p$) of this order completely, and we study the $p$-part in a particular (but still reasonably general) situation.}}
\tableofcontents
\section{Introduction}
Let $k$ be a field and let $G$ be a smooth, connected, and commutative algebraic group over $k.$ If $k$ is perfect, then $G$ fits into an exact sequence
$$0\to H\to G\to A\to 0$$ over $k$ such that $H$ is a smooth connected commutative affine algebraic group and $A$ an Abelian variety over $k,$ respectively. This is a consequence of Chevalley's theorem (see \cite{Con} for a proof of this fact in the language of schemes, and for more background). However, it is well-known that Chevalley's theorem fails completely if $k$ is not perfect; we shall see many such examples in this paper. To deal with this phenomenon, Totaro \cite{T} recently introduced the notion of \it pseudo-Abelian varieties \rm and worked out much of their structure. A pseudo-Abelian variety over the field $k$ is a smooth connected commutative\footnote{This is not part of Totaro's original definition but a consequence of it; we shall recall those details later.} algebraic group which does not admit any smooth connected affine closed algebraic subgroups. This notion works very well in the context of Chevalley's theorem, even over imperfect fields. Indeed, every smooth connected (not necessarily commutative) algebraic group (over an arbitrary field) is an extension of a pseudo-Abelian variety by a linear algebraic group in a unique way. Over a perfect field, all pseudo-Abelian varieties are Abelian by Chevalley's theorem, but over a non-perfect field, pseudo-Abelian varieties which are not Abelian are ubiquitous (see \cite{T} for more details).

Now let $\Og_K$ be an excellent discrete valuation ring with field of fractions $K$ and residue field $\kappa.$ The field $K$ is perfect if and only if it is of characteristic 0, so if $p:=\mathrm{char}\, K>0$, then there are plenty of pseudo-Abelian varieties over $K$ which are not Abelian varieties. In this paper we shall study \it degenerations \rm of pseudo-Abelian varieties, that is, smooth separated models $\mathscr{P}'\to \Spec \Og_K$ of $P$ of finite type. Even if $P$ is an Abelian variety, it is not in general possible to predict the behaviour of a general model of this kind. However, it has been known for a long time that Abelian varieties over discretely valued fields admit \it Néron models, \rm which are smooth separated models of finite type that are characterized by a universal property. As it turns out, a classical criterion for the existence of Néron models implies immediately that pseudo-Abelian varieties over fields of fractions of \it excellent \rm discrete valuation rings admit Néron models as well. Since it is almost impossible to exaggerate the role played by Néron models in the study of Abelian varieties, it seems very natural to study Néron models of pseudo-Abelian varieties in more detail in order to understand the behaviour of those objects under degeneration. We shall seek to begin such a study in this paper. 

In the world of Abelian varieties, one can use Néron models to distinguish between \it good reduction, semiabelian reduction, and additive reduction. \rm We shall introduce analogous notions for pseudo-Abelian varieties; these are defined purely in terms of numerical invariants of algebraic groups (such as toric and Abelian ranks) and are straightforward generalizations of the corresponding notions in the case of Abelian varieties. For Abelian varieties, it is well-known that the the various reduction types are cohomological invariants, i.e., they only depend upon the first $\ell$-adic cohomology (or, equivalently, the $\ell$-adic Tate module) of the Abelian variety; here we have to choose a prime number $\ell$ invertible in $\Og_K.$

Our first main result is that this still holds for pseudo-Abelian varieties:

\begin{theorem} \rm (Theorems \ref{NOScritthhm} and \ref{unipotentthm}) \it Let $P$ be a pseudo-Abelian variety over $K.$ Then $P$ has good reduction (resp. pseudo-semiabelian reduction)\footnote{See Definitions \ref{goodreddef} and \ref{pseudosemabdef}.} if and only if the Galois representation on $T_\ell(P)$ is unramified (resp. all elements of an inertia subgroup act as unipotent operators).
\end{theorem}
Although Néron models generally behave badly with respect to base change, the situation is somewhat better if the Abelian variety in question has good or semiabelian reduction; in the first case, the Néron model commutes with any faithfully flat base change of discrete valuation rings and in the second case, at least the identity component of the Néron model will commute with general faithfully flat base change (even though the Néron model itself will usually not have this property). This follows from the fact that a semi-Abelian model of an Abelian variety is uniquely determined up to unique isomorphism, and it implies that the property of having good (resp. semiabelian) reduction is not affected by passing to a finite (possibly ramified) extension of $K.$  We shall give examples to show that the uniqueness properties just mentioned do not hold for pseudo-Abelian varieties in general. However, the representation-theoretic criteria will allow us to deduce that the property of having good (resp. pseudo-semiabelian) reduction is not lost after passing to an arbitrary finite separable extension.
Another invariant attached to a pseudo-Abelian variety $P$ over $K$ by means of its Néron model $\mathscr{P}\to \Spec \Og_K$ is the \it order of the group scheme $\Phi:=\mathscr{P}_{\kappa}/\mathscr{P}^0_{\kappa},$ \rm which is usually known as the group scheme of connected components of $\mathscr{P}_{\kappa}.$ We shall see that, if $\ell$ is a prime number invertible on $\Og_K,$ the order of the $\ell$-part of $\Phi$ is controlled completely by the analogous invariant associated with the maximal Abelian subvariety of $P.$ These methods do not work if $\ell =p,$ and the $p$-part of $\Phi$ remains largely mysterious (see Question \ref{Phiquestion}). However, we are able to show that it vanishes for a particular class of pseudo-Abelian varieties with good reduction constructed in \cite{T}. \\
Throughout the paper, we shall use the following notation:
\begin{itemize}
\item[$\cdot$] $\Og_K$ is an excellent discrete valuation ring of equal characteristic $p>0.$
\item[$\cdot$]$K:=\Frac \Og_K.$
\item[$\cdot$]$\langle\pi\rangle=\mathfrak{m}\subseteq \Og_K$ is the maximal ideal of $\Og_K.$ 
\item[$\cdot$]$\kappa:=\Og_K/\mathfrak{m}$ is the residue field of $\Og_K$ (which is not assumed to be perfect unless stated otherwise).
\item[$\cdot$]For morphisms of schemes $X\to S$ and $S'\to S,$ we let $X_{S'}:=X\times_SS'.$ If $S=\Spec \Og_K$ and $S'$ is the spectrum of an $\Og_K$-algebra $B,$ we let $X_B:=X_{S'}.$
\item[$\cdot$]If $\Og_K$ is Henselian and $X\to \Spec \Og_K$ is a quasi-finite morphism, we let $X=X^{\mathrm{f}} \sqcup X^{\eta}$ be the decomposition of $X$ into an $\Og_K$-finite open subscheme $X^{\mathrm{f}}$ and an $\Og_K$-scheme $X^{\eta}$ with empty special fibre. This decomposition is functorial in $X$ (see Proposition \ref{decompprop}),
\item[$\cdot$] For a group scheme $G$ locally of finite presentation over a field $k$, we denote by $G^0$ the component of the unit section of $G.$ 
\item[$\cdot$] For a smooth group scheme $\mathscr{G}\to \Spec \Og_K$ with connected generic fibre, we let $\mathscr{G}^0$ be the complement (in $\mathscr{G}$) of the union of irreducible components of $\mathscr{G}_{\kappa}$ which do not contain the unit section. This is the unique open subgroup scheme of $\mathscr{G}$ over $\Og_K$ with connected special fibre.
\item[$\cdot$] For a field $k,$ we let $k\sep$ (resp. $k\alg$) denote a separable (resp. algebraic) closure of $k.$ Moreover, we let $k\perf$ denote the perfect closure of $k,$ i.e., the unique perfect purely inseparable algebraic extension of $k.$ 
\end{itemize}
\section{Some technical remarks}
\subsection{Weil restriction}
Let $S'\to S$ be a finite and locally free morphism of schemes. Let $X$ be a scheme over $S'.$ Then the functor $\Res_{S'/S}X$ is defined to be the pushforward of the functor of points of $X$ along $S'\to S.$ This is clearly still a sheaf in the fppf-topology (and therefore, \it a forteriori, \rm in the étale and Zariski topologies as well). However, in order to ensure that this functor is representable, one must generally impose a condition on the morphism $X\to S$ concerning common affine neighbourhoods of finite sets of points contained in the fibres of this map. The following proposition shows that this condition can be dropped if $S'\to S$ is a universal homeomorphism. This has already been observed in \cite{Bert}, Corollary A.5 (and it is also implicit in many other places in the literature). We give a more direct proof below. This technical point will turn out to be important in this paper, since it will allow us to perform Weil restrictions along finite locally free universal homeomorphisms without having to rely on Raynaud's deep results on the quasi-projectivity of group schemes. 

\begin{proposition}
Let $S'\to S$ be a morphism of schemes which is finite and locally free. Furthermore assume that $S'\to S$ be a universal homeomorphism (i.e., a morphism such that for all $S$-schemes $T,$ the map $T\times_SS'\to T$ is a homeomorphism). Let $X$ be a scheme over $S.$ Then $\Res_{S'/S}X$ is representable. \label{universalhomeoprop}
\end{proposition}
\begin{proof} We may assume without loss of generality that $S$ (and hence $S'$) is affine. Let $(X_i)_{i\in I}$ be an open affine cover of $X$ indexed by some set $I.$ Then we have, for each $i\in I,$ a canonical morphism of functors 
$$\Res_{S'/S} X_i\to \Res_{S'/S}X,$$ which is representable by open immersions (\cite{BLR}, Chapter 7.6, Proposition 2(i)). Furthermore, we know that the functors $\Res_{S'/S}X_i$ are representable (\cite{BLR}, Chapter 7.6, first part of the proof of Theorem 4). Hence all that remains to be shown is that the set of functors $\{\Res_{S'/S} X_i\colon i\in I\}$ covers $\Res_{S'/S} X.$ To see this, let $T\to S$ be a morphism of schemes, and let $\xi\in \Res_{S'/S}X(T).$ Then $\xi$ is a morphism $T\times_SS'\to X.$ Now let $i\in I.$ Because $S'\to S$ is a universal homeomorphism, we can find a unique open subscheme $T_i\subseteq T$ such that 
$$T_i\times_SS'=\xi^{-1}(X_i).$$ If we let $\xi_i:=\xi\mid_{T_i},$ we see that $\xi_i\in \Res_{S'/S}X_i(T_i).$ This concludes the proof.  
\end{proof}\\
The Proposition allows us to prove a strong representability result in the world of group schemes of finite type over fields: 
\begin{proposition}
Let $k$ be a field and let $A$ be a finite $k$-algebra. Let $G\to \Spec A$ be a (not necessarily smooth) group scheme of finite type. Then $\Res_{A/k}G$ is representable by a quasi-projective group scheme over $k.$ \label{finitealgprop}
\end{proposition}
\begin{proof}
Because Galois descent is effective for quasi-projective schemes, we may assume without loss of generality that $A$ be local and remain so after base change to $k\sep.$ Indeed, over $k\sep,$ $A$ splits into finitely many finite local $k\sep$-algebras, and we can simply replace $k$ by a finite Galois extension over which this decomposition is defined. Now observe that, if $\mathfrak{m}\subseteq A$ denotes the unique maximal ideal, the extension $k\subseteq A/\mathfrak{m}$ is purely inseparable (since $\Spec (A/\mathfrak{m}\otimes_kk\sep)$ is a closed subscheme of the one-point scheme $\Spec A\otimes_kk\sep$).  This allows as to deduce that the map $\Spec A\to \Spec k$ is a universal homeomorphism, which can be seen by considering the morphisms $\Spec A/\mathfrak{m}\to \Spec A\to \Spec k,$ and observing that the first morphism is a universal homeomorphism (as $\mathfrak{m}$ is nilpotent) as well as the composition (because $A/\mathfrak{m}$ is a purely inseparable extension of $k$). Hence the claim follows from Proposition \ref{universalhomeoprop}, together with the fact that group schemes of finite type over fields are always quasi-projective (see \cite{CGP}, Proposition A.3.5). 
\end{proof}\\
$\mathbf{Remark}.$ A similar result has been proven in \cite{CGP} (see \it op. cit., \rm Proposition A.5.1), where the $k$-algebra $A$ is assumed to be reduced. The Proposition above shows that this hypothesis is redundant.

Finally, we recall recall that the small étale site of a scheme is topologically invariant, as made precise by the following result of Grothendieck:
\begin{proposition}
Let $f\colon S'\to S$ be a universal homeomorphism of schemes. Then the categories $S_{\et}$ and $S'_{\et}$ of schemes étale over $S$ (resp. $S'$) are identified via the functors $f^\ast=-\times_SS'$ and $f_\ast=\Res_{S'/S}-,$ which are mutually inverse equivalences of categories. Moreover, these equivalences are compatible with the étale topologies on $S_{\et}$ and $S'_{\et}.$ \label{topolequivprop}
\end{proposition}  
\begin{proof}
See \cite{SGA4 VIII}, Théorème 1.1.
\end{proof} 
\subsection{Excellent discrete valuation rings}
As mentioned in the introduction, we shall always suppose that the base ring $\Og_K$ over which we work be an excellent discrete valuation ring of equal characteristic $p>0.$ This restriction is crucial because otherwise we would not be able to guarantee that pseudo-Abelian varieties over $K$ admit Néron models over $\Og_K.$ Recall that a Noetherian ring $R$ ring is said to be \it excellent \rm if it satisfies the following three conditions: \\
\\
(i) For each prime ideal $\mathfrak{p}\subseteq R,$ the map $\Spec \widehat{R}_{\mathfrak{p}}\to \Spec R_{\mathfrak{p}}$ is flat and has geometrically regular fibres, where $ \widehat{R}_{\mathfrak{p}}$ denotes the $\mathfrak{p}$-adic completion of $R_{\mathfrak{p}}$,\\
(ii) Each integral scheme $T$ finite over $R$ contains a regular dense open subset, and \\
(iii) $R$ is universally catenary. \\
\\
In fact, there are various equivalent definitions; see \cite{Raynaud} for more details. The following Lemma is well-known, but its proof is spread across many sources, so we recall a proof for the reader's convenience: 
\begin{lemma}
Let $R$ be a discrete valuation ring with field of fractions $K.$ Let $\widehat{R}$ be the completion of $R$ and let $\widehat{K}:=\Frac\widehat{R}.$ Then $R$ is excellent of and only if the $K$-algebra $\widehat{K}$ is geometrically reduced, i.e., if and only if the extension $K\subseteq \widehat{K}$ is separable. 
\end{lemma}
\begin{proof}
Since the map $R\to \widehat{R}$ induces an isomorphism on residue fields, our assumption ensures that condition (i) is satisfied, and it is obvious that our condition is necessary. Hence it suffices to show that conditions (ii) and (iii) are satisfied for all discrete valuation rings. Condition (ii) is satisfied because all regular local rings are universally catenary (see \cite{PS}, Exemple 1.1.3). For condition (iii) it suffices to show that for all finite maps $T\to \Spec R$ with $T$ integral, the set of regular points of $T$ contains a non-empty open subset. This can be reduced to the case where the base ring is a field, in which case it is clear. 
\end{proof}\\
Excellent discrete valuation rings have the following important property:
\begin{proposition}
Let $R$ be an excellent discrete valuation ring and let $L$ be a finite (not necessarily separable) extension of the field of fractions $K$ of $R.$ Then the integral closure of $R$ in $L$ is finitely generated as an $R$-module. In other words, excellent discrete valuation rings are Japanese. \label{japaneseprop}
\end{proposition}
\begin{proof}
This is elementary if $R$ is complete (see, for example, the Remark in \cite{Neu} after the proof of Proposition 6.8). For general excellent $R$, let $R'$ be the integral closure of $R$ in $L.$ Then $R'\otimes_R\widehat{R}$ is canonically an $\widehat{R}$-submodule of the integral closure of $\widehat{R}$ in $L\otimes_K\widehat{K}.$ Because the extension $K\subseteq \widehat{K}$ is separable, the $\widehat{K}$-algebra $L\otimes_K\widehat{K}$ is reduced, and hence a finite product of finite field extensions of $\widehat{K}.$ This shows that $R'\otimes_R\widehat{R}$ is finite over $\widehat{R},$ and since $\Spec \widehat{R}\to \Spec R$ is an fpqc-cover, we find that $R'$ is finite over $R.$  
\end{proof}\\
Although non-excellent discrete valuation rings exist (see \cite{Raynaud}, Proposition 11.6), they are rather rare. For example, every discrete valuation ring which arises as the local ring of a normal scheme of finite type over a field at a point of codimension one is excellent (\cite{Raynaud}, Théorème 5.1). 

\section{Pseudo-Abelian varieties and virtual ranks}
Let $k$ be a field and let $G$ be a smooth connected commutative algebraic group over $K.$ If $k$ is perfect, then $G$ fits into an exact sequence $0\to H \to G\to A\to 0,$ where $H$ and $A$ are a smooth connected affine algebraic group and $A$ an Abelian variety, respectively. Furthermore, there exist a unipotent group $U$ and a torus $T$ such that $H=U\times_kT.$ If $k$ is not perfect, then neither of those two statements holds true in general. This motivates the following
\begin{definition}
Let $P$ be a smooth, connected and commutative algebraic group over a field $k.$ Then $P$ is \rm pseudo-Abelian \it if the maximal smooth connected affine closed subgroup of $P$ is trivial.  
\end{definition}
$\mathbf{Remark.}$ This definition is due to Totaro (\cite{T}, Definition 0.1). In fact, our definition is slightly different from that of \it loc. cit., \rm where $P$ is not assumed to be commutative. The two definitions are, however, equivalent by \cite{T}, Theorem 2.1.\\
\\
If $G$ is any smooth, connected commutative algebraic group over a field $k,$ then $G$ fits into an exact sequence 
$$0\to H\to G\to P\to 0$$ with $H$ affine and $P$ pseudo-Abelian in a unique way. Moreover, pseudo-Abelian varieties which are not Abelian exist over any non-perfect field (\cite{T}, p. 649). On the other hand, any pseudo-Abelian variety over a perfect field is Abelian by Chevalley's theorem. \\
The following numerical invariants will be of particular importance in this paper: 
\begin{definition}
Let $G$ be a smooth commutative group scheme over a field $k.$ Suppose first that $G$ be connected. By Chevalley's theorem (\cite{Con}, Theorem 1.1), there exists a unique exact sequence
$$0\to H\to G\times_k\Spec k\perf\to A\to 0$$ with $H$ affine and $A$ Abelian. Furthermore, we can write $H=U\times_{k\perf} T,$ where $U$ is a smooth connected unipotent group and $T$ a Torus over $k\perf,$ respectively (see \cite{SGA3}, Exposé XVII, Théorème 7.2.1 b)). \\
We define the \rm virtual Abelian rank \it $\alpha(G)$ of $G$ to be the dimension of $A,$ and the \rm virtual unipotent rank \it $y(G)$ to be the dimension of $U.$ Furthermore, we define the \rm toric rank \it $t(G)$ of $G$ to be the dimension of $T.$ \\
Finally, if $G$ is not necessarily connected, we let $\alpha(G):=\alpha(G^0)$ and similarly for $y(G),$ $t(G).$ 
\end{definition}
$\mathbf{Remark.}$ The toric rank of $G$ is not "virtual" because by Grothendieck's theorem on tori (\cite{SGA3}, Exposé XIV, Théorème 1.1), there exists a unique closed subtorus $T'$ of $G$ over $k$ such that $T=T'\times_k\Spec k\perf.$ One verifies easily that the invariants $\alpha(G),$ $y(G)$ and $t(G)$ are invariant under any extension of the base field.  \\
\\
If $P$ is a pseudo-Abelian variety over a field $k$ which is not Abelian, then $\alpha(P)$ and $y(P)$ are both strictly positive. Indeed, if $\alpha(P)=0$ then $P$ is affine, and if $y(P)=0$ then $P$ is Abelian. \\
\\
For smooth connected group schemes $G$ over a field $k,$ the invariants $\alpha(G)$ and $t(G)$ are encoded in a strong way by the $\ell$-torsion subschemes: 
\begin{proposition}
Let $G_1$ and $G_2$ be smooth connected commutative group schemes over a field $k.$ Suppose there exists a morphism of $k$-group schemes $f\colon G_1\to G_2$ such that, for infinitely many prime numbers $\ell$ invertible in $k,$ the induced morphism $f[\ell]\colon G_1[\ell]\to G_2[\ell]$ is an isomorphism. Then $\alpha(G_1)=\alpha(G_2)$ and $t(G_1)=t(G_2).$ \label{alphainvariantprop}
\end{proposition}
\begin{proof}
Since the invariants in question do not change when we replace $k$ by a bigger field, we may assume without loss of generality that $k$ be algebraically closed. Then there exists a commutative diagram
$$\begin{CD}
0@>>>T_1\times_kU_1@>>> G_1@>>> A_1@>>> 0\\
&&@V{f_{\mathrm{aff}}}VV@V{f}VV@VV{f}_{\mathrm{ab}}V\\
0@>>>T_2\times_kU_2@>>> G_2@>>> A_2@>>> 0
\end{CD}$$
with exact rows, where the $T_j$ and $U_j$ are tori and smooth commutative connected unipotent groups over $k,$ respectively, and the $A_j$ Abelian varieties (see \cite{Con}, Theorem 1.1 and \cite{SGA3}, Exposé XVII, Théorème 7.2.1 b)). \\
First observe that we may also assume without loss of generality that $f_{\mathrm{aff}}$ be a closed immersion. Indeed, if the scheme-theoretic intersection $\ker f_{\mathrm{aff}}\cap T_1$ were not finite, then $\ker f_{\mathrm{aff}}$ would contain non-trivial $\ell$-torsion points for all prime numbers $\ell\in k^\times,$ contradicting our assumption on $f.$ Hence, replacing $G_1$ by $G_1/\ker f_{\mathrm{aff}}$ does not alter the toric rank, and it clearly does not alter the virtual Abelian rank either.
We must show that the induced morphism 
$$G_1[\ell]\to (G_1/\ker f_{\mathrm{aff}})[\ell]$$ is an isomorphism for infinitely many prime numbers $\ell\in k^\times.$ We claim that this is the case for all $\ell$ invertible in $k$ for which $f[\ell]$ is an isomorphism and which do not divide the number of irreducible components of $\ker f_{\mathrm{aff}}.$ Indeed, by the snake lemma, it suffices to show that the map 
$$[\ell]\colon \ker f_{\mathrm{aff}} \to \ker f_{\mathrm{aff}}$$ is fppf-surjective for all such $\ell.$ Since $k$ is perfect, $(\ker f_{\mathrm{aff}})_{\mathrm{red}}^0$ is a group scheme over $k,$ which is smooth by construction. Moreover, the quotient $(\ker f_{\mathrm{aff}})^0/(\ker f_{\mathrm{aff}})_{\mathrm{red}}^0$ is an infinitesimal finite group scheme over $k,$ so multiplication by $\ell$ is an isomorphism on that group scheme for all such $\ell.$ Moreover, multiplication by $\ell$ is étale on $(\ker f_{\mathrm{aff}})_{\mathrm{red}}^0,$ so $[\ell]$ is surjective on $(\ker f_{\mathrm{aff}})_\mathrm{red}^0.$ The snake lemma now shows that $[\ell]$ is an isomorphism on $(\ker f_{\mathrm{aff}})^0.$ Finally, a completely analogous argument applied to the sequence
$$0\to (\ker f_{\mathrm{aff}})^0\to (\ker f_{\mathrm{aff}})\to Q\to 0$$ (where $Q$ is defined to be the quotient) proves the claim, noting that $\ell$ does not divide the order of $Q$ by assumption. In particular, we find that the morphism 
$$(G_1/(\ker f_{\mathrm{aff}}))[\ell]\to G_2[\ell]$$ is an isomorphism for infinitely many $\ell\in k^\times.$\\
Since $k$ is perfect, $(\ker f_{\mathrm{ab}})_{\mathrm{red}}$ is a group scheme over $k,$ and it is clearly of finite type. Let $\ell$ be a prime number invertible in $k$ which does not divide the number of irreducible components of this group scheme, and such that $f[\ell]$ is an isomorphism. Then we obtain a commutative diagram
$$\begin{CD}
0@>>>T_1[\ell]@>>> G_1[\ell]@>>> A_1[\ell]@>>> 0\\
&&@V{f_{\mathrm{aff}}[\ell]}VV@V{f[\ell]}VV@VV{{f}_{\mathrm{ab}}[\ell]}V\\
0@>>>T_2[\ell]@>>> G_2[\ell]@>>> A_2[\ell]@>>> 0.
\end{CD}$$
We shall prove that $f_{\mathrm{ab}}[\ell]$ and $f_{\mathrm{aff}}[\ell]$ are both isomorphisms; this will clearly imply our claim. Using the snake lemma, all we have to show is that the morphism $\ker(f_{\mathrm{ab}}[\ell])\to \mathrm{coker}(f_{\mathrm{aff}}[\ell])$ vanishes. Observe that
$$\ker(f_{\mathrm{ab}}[\ell]) \subseteq (\ker f_{\mathrm{ab}})_{\mathrm{red}}[\ell]=(\ker f_{\mathrm{ab}})_{\mathrm{red}}^0[\ell].$$ Using the naturality of the snake morphism, we obtain a commutative diagram
$$\begin{CD}
\ker(f_{\mathrm{ab}}[\ell])@>>>\mathrm{coker}(f_{\mathrm{aff}}[\ell])\\
@VVV@VVV\\
(\ker f_{\mathrm{ab}})_{\mathrm{red}}^0@>>>\mathrm{coker} f_{\mathrm{aff}}.
\end{CD}$$
Since $f_{\mathrm{aff}}$ is a closed immersion, we have $\mathrm{coker}(f_{\mathrm{aff}})[\ell]=\mathrm{coker}(f_{\mathrm{aff}}[\ell]),$ which implies that the right vertical arrow in the diagram is a closed immersion; the same holds for the left vertical arrow in any case. However, $(\ker f_{\mathrm{ab}})_{\mathrm{red}}^0$ is an Abelian variety and $\mathrm{coker}(f_{\mathrm{aff}})$ is affine, so the bottom arrow vanishes. Hence so does the top arrow and we conclude. 
\end{proof}\\
$\mathbf{Remark.}$ Clearly, the corresponding statement for $y(-)$ does not hold, as is shown by the map $\mathbf{G}_{\mathrm{a}}\to \mathbf{G}_{\mathrm{a}}^2;$ $x\mapsto (x,x).$\\
\\
The preceding Proposition will have several important applications in this paper. 
\begin{lemma}
Let $\Og_K$ be a discrete valuation ring with field of fractions $K.$ Let $\mathscr{U}\to \Spec \Og_K$ be a flat separated commutative group scheme such that $U:=\mathscr{U}_K$ is annihilated by a power of $p=\mathrm{char}\, K >0 $ (which is the case if, for example, $U$ is unipotent). Let $\ell$ be a prime number invertible in $\Og_K.$ Then the map $[\ell]\colon \mathscr{U}\to \mathscr{U}$ is an isomorphism. \label{lisomlem}
\end{lemma}
\begin{proof}
Choose $N\in \N$ such that $[p^N]=0$ on $U.$ Choose a natural number $m$ such that $m\ell\equiv 1 \mod p^N.$ Then $[m]\circ[\ell]=[\ell]\circ[m]=\mathrm{Id}_U$ on $U.$ Hence the same holds for $\mathscr{U}$ because $\mathscr{U}\to \Spec \Og_K$ flat and is separated.  
\end{proof}
\begin{proposition}
Let $\Og_K$ be an excellent discrete valuation ring with field of fractions $K$ and residue field $\kappa.$  Assume $\mathrm{char}\,K>0.$ Let $\mathscr{P}'\to \Spec \Og_K$ be a smooth separated group scheme whose generic fibre $P$ is a pseudo-Abelian variety over $K.$ Then \\
(i) $\alpha(\mathscr{P}'_\kappa)\leq \alpha(P)$ and \\
(ii) $t(\mathscr{P}'_\kappa)\leq \alpha(P)-\alpha(\mathscr{P}'_\kappa).$ \label{deltajumpprop}
\end{proposition}
\begin{proof}
Because $\alpha(-)$ and $t(-)$ are invariant under extensions of the base field, we may assume without loss of generality that $\Og_K$ be Henselian. Let $L$ be a finite, purely inseparable extension of $K$ such that $\mathscr{P}'_\kappa$ is an extension of an Abelian variety by a smooth, connected unipotent group. Since $\Og_K$ is excellent (which remains unchanged after replacing $\Og_K$ with its Henselization by \cite{Raynaud}, Théorème 8.1(iv)), the integral closure of $\Og_L$ in $L$ is a discrete valuation ring. We may replace $\Og_K$ by this integral closure and assume that $U$ is already defined over $K.$ Furthermore, we may replace $\mathscr{P}'$ by its identity component $\mathscr{P}'^0$ and hence assume that $\mathscr{P}'\to \Spec \Og_K$ be of finite type. Let $\mathscr{U}$ be the Zariski closure of $U$ in $\mathscr{P}'.$ Then $\mathscr{U}$ is a flat and separated model of $U.$ Consider the exact sequence of fppf-sheaves
$$0\to \mathscr{U}\to \mathscr{P}'\to \mathscr{P}'/\mathscr{U}\to 0. \label{sequence1}$$ The fourth sheaf which appears in this sequence is representable by \cite{An}, Chapitre IV, Théorème 4.C. Since the morphism $\mathscr{U}\to \Spec \Og_K$ is fppf by construction, so is the morphism $\mathscr{P}'\to \mathscr{P}'/\mathscr{U}.$ This shows that the group scheme $\mathscr{P}'/\mathscr{U}$ is flat and of finite presentation over $\Og_K.$ The fibres of the quotient scheme are clearly smooth over the corresponding residue fields (since they are geometrically reduced), so the morphism $\mathscr{P}'/\mathscr{U}\to \Spec \Og_K$ is smooth. Now the snake lemma tells us, together with Lemma \ref{lisomlem}, that the maps
$\mathscr{P}'[\ell]\to \mathscr{P}'/\mathscr{U}[\ell]$ are isomorphisms for all prime numbers $\ell$ invertible in $\Og_K.$ Hence Proposition \ref{alphainvariantprop} implies that $\alpha(\mathscr{P}'_K)=\alpha((\mathscr{P}'/\mathscr{U})_K)$ and $\alpha(\mathscr{P}'_\kappa)=\alpha((\mathscr{P}'/\mathscr{U})_\kappa),$ which reduces our claim to the case where the generic fibre of $\mathscr{P}'$ is an Abelian variety. However, in this case, the assertion is clear for dimension reasons. This proves (i); assertion (ii) follows from an entirely analogous argument. 
\end{proof}\\
This Proposition shows that the virtual Abelian rank is relatively well-behaved in the situation which interests us. Note that the toric rank can jump both up and down, even when the base scheme is the spectrum of a compete discrete valuation ring and both fibres are semiabelian. 

\subsection{The Tate module}
Let $\ell$ be a prime number invertible in $k,$ and let $P$ be a pseudo-Abelian variety over $k.$ As usual, we define the \it Tate module \rm of $P$ as
$$T_{\ell}(P):=\varprojlim P[\ell^n](k\sep).$$ This is a $\Z_\ell$-module with a natural action of $\Gal(k\sep/k).$ In order to understand the structure of $T_{\ell}(P),$ we need the following
\begin{lemma}
Let $P$ be a pseudo-Abelian variety over a field $k.$ Then there exists a finite, purely inseparable extension $L/k$ together with an exact sequence
$$0\to H\to P_L\to A\to 0,$$ where $H$ and $A$ are a smooth connected commutative unipotent algebraic group over $L$ and an Abelian variety over $L,$ respectively. Furthermore, given such an $L,$ the exact sequence is uniquely determined by $P.$ \label{Psplitlem}
\end{lemma}
\begin{proof}
Let $k\perf $ be the perfect closure of $k.$ Since $k\perf$ is perfect, we may apply Chevalley's theorem to $P\times_k\Spec k\perf$, so we obtain an exact sequence $0\to H'\to P\times_k\Spec k\perf\to A'\to 0,$ where $H'$ and $A'$ are a commutative smooth connected affine algebraic group and an Abelian variety over $k\perf,$ respectively. Since all schemes appearing in this sequence are of finite presentation over $k\perf,$ the sequence descends to a finite subextension $L$ of $k\subseteq k\perf,$ which must be purely inseparable. Now suppose that, over a field $L$ as in the Lemma, we have another such exact sequence $0\to H''\to P_L\to A''\to 0.$ Then $H''$ maps into $H'$ and the cokernel of this morphism is isomorphic to the kernel of the induced map $A''\to A,$ which implies that it must be proper over $k.$ Since the cokernel is also smooth, connected, and affine, it follows that it is trivial. Now all that remains to be shown is that $H$ is unipotent, which can be checked after base change to $k\perf.$ We can find a torus $T$ and a unipotent group $U$ over $k\perf$ such that $H\times_k\Spec k\perf=T\times_{k\perf}U.$ By a theorem of Grothendieck (\cite{SGA3}, Exposé XIV, Théorème 1.1), the maximal torus $T$ of $P\times_k\Spec k\perf$ would be defined over $k,$ contradicting our assumption that $P$ be pseudo-Abelian. 
\end{proof}\\
This Lemma allows us to give a precise description of the structure of $T_{\ell}(P)$:
\begin{proposition}
Let $P$ be a pseudo-Abelian variety over a field $k,$ and let $\ell$ be a prime number invertible in $k.$ Then $T_{\ell}(P)$ is a finitely generated free $\Z_{\ell}$-module whose rank equals $2\alpha(P).$ 
\end{proposition}
\begin{proof}
Choose a finite purely inseparable extension $L/k$ and an exact sequence $0\to H\to P_L\to A\to 0$ as in Lemma \ref{Psplitlem}. Since the map $[\ell^n]\colon H\to H$ is an isomorphism (as $H$ is unipotent, smooth, and connected), the snake lemma tells us that the morphism $P_L[\ell^n]\to A[\ell^n]$ is an isomorphism. Hence we find 
$$T_{\ell}(P_L)=T_{\ell}(A).$$ However, by topological invariance of the étale site (see Proposition \ref{topolequivprop}), the Galois modules $T_{\ell}(P)$ and $T_{\ell}(P_L)$ are canonically identified, so we obtain our claim. 
\end{proof}

\subsection{Néron models}
Let $R$ be a discrete valuation ring with field of fractions $K$. Suppose further that $\mathscr{G}\to \Spec R$ be a smooth separated group scheme. Recall that $\mathscr{G}\to \Spec R$ is a \it Néron model\footnote{What we define here is called a \it Néron lft-model \rm in \cite{BLR}; in \it op. cit. \rm Néron models are of finite type over $R$ by definition. The term \it Néron lft-model \rm is still used sometimes to emphasize that general smooth separated models are allowed.} \rm of its generic fibre if it satisfies the following universal property: For each smooth morphism $T\to \Spec R$ and each $K$-map $\varphi\colon T_K\to \mathscr{G}_K,$ there exists a unique $R$-map $T \to \mathscr{G}$ which extends $\varphi.$\\
Now let $\Og_K$ be an \it excellent \rm discrete valuation ring with residue field $\kappa,$ maximal ideal $\mathfrak{m}=\langle \pi\rangle$ and field of fractions $K,$ as before. We shall always assume that the characteristic $p$ of $K$ be positive. In this case, the field $K$ is never perfect (since no uniformizer of $\Og_K$ is contained in the image of the Frobenius), so there will be plenty of pseudo-Abelian varieties over $K$ which are not Abelian varieties. Whenever $L$ is a finite extension of $K$ (separable or not), we shall denote the integral closure of $\Og_K$ in $L$ by $\Og_L;$ this is a finite extension of $\Og_K$ since $\Og_K$ is excellent. 
\begin{proposition}
Let $P$ be a pseudo-Abelian variety over $K.$ Then $P$ admits a Néron model $\mathscr{P}$ over $\Og_K.$ Moreover, $\mathscr{P}$ is of finite type over $\Og_K.$ 
\end{proposition} 
\begin{proof} Let $K^{\mathrm{sh}}$ be the field of fractions of a strict Henselization of $\Og_K.$ Then the extension $K\subseteq K^{\mathrm{sh}}$ is separable, so $P_{K^{\mathrm{sh}}}$ is pseudo-Abelian (\cite{T}, Lemma 2.3). In particular, $P_{K^{\mathrm{sh}}}$ does not contain any closed subgroups isomorphic to $\Gm$ or $\mathbf{G}_{\mathrm{a}}.$ Therefore the claim follows from \cite{BLR}, Chapter 10.2, Theorem 1(b'). 
\end{proof}
\begin{definition}
Let $P$ be a pseudo-Abelian variety over $K$ with Néron model $\mathscr{P}$ over $\Og_K.$ Let $\mathscr{P}'\to \Spec \Og_K$ be another smooth separated model of finite type of $P.$ We define the \rm defect \it $\delta(\mathscr{P}')$ of the model $\mathscr{P}'$ to be 
$$\delta(\mathscr{P}'):=\alpha(P)-\alpha(\mathscr{P}'_\kappa).$$
Moreover, we define the \rm defect $\delta(P)$ of \it $P$ to be the defect of the Néron model $\delta(\mathscr{P}).$  
\end{definition}
By Proposition \ref{deltajumpprop}, $\delta(\mathscr{P}')\geq 0$ for all smooth group schemes $\mathscr{P}'\to \Spec \Og_K$ with pseudo-Abelian generic fibre. In particular, $\delta(P)\geq 0$ for all pseudo-Abelian varieties $P$ over $K.$

\section{The Néron-Ogg-Shafarevich criterion for pseudo-Abelian varieties}
A fundamental result about Néron models of Abelian varieties is that the information about whether the Abelian variety has good reduction is completely encoded the Galois representation on the Abelian variety's Tate module. In this chapter, we shall prove an analogous result for pseudo-Abelian varieties. First of all, we must define what \it good reduction \rm should mean for pseudo-Abelian varieties which are not Abelian. 
\begin{definition}
Let $\Og_K$ be a discrete valuation ring with field of fractions $K.$ Let $P$ be a pseudo-Abelian variety over $K.$ Then $P$ has \rm good reduction \it if and only if $\delta(P)=0,$ i.e., if and only if the virtual Abelian rank is constant in the family $\mathscr{P}\to \Spec \Og_K.$ \label{goodreddef}
\end{definition}
This definition coincides with the usual one of $P$ is Abelian. Before we state our main result, we need one more technical preparation. Fix an excellent discrete valuation ring $\Og_K$ with residue field $\kappa$ and field of fractions $K.$ Assume $p:=\mathrm{char}\, K>0.$ 
\begin{lemma}
(i) Let $R\subseteq S$ be a finite extension of (not necessarily excellent) discrete valuation rings with the property that the induced extension of fields of fractions is purely inseparable. Then the morphism 
$\Spec S\to \Spec R$ is a universal homeomorphism. \\
(ii) Suppose that $R$ be excellent and let $S$ be the integral closure of $R$ in a finite purely inseparable extension of the field of fractions of $R.$ Then $S$ is a discrete valuation ring and the extension $R\subseteq S$ is finite.  \label{univlem}
\end{lemma}
\begin{proof}
(i) The statement is clear if the field of fractions of $R$ has characteristic zero, so we shall assume that the characteristic $p$ of this field be positive. A morphism between affine schemes is a universal homeomorphism of and only if it is bijective, induces purely inseparable extensions on residue fields, and the corresponding map on rings is integral (see \cite{EGA IV}, Proposition 2.4.5). Integrality and bijectivity are clear. Hence all that remains to be shown is that the induced extensions of residue fields at the special points is purely inseparable. Let $x$ be an element of the residue field of $S.$ Choose an element $y$ of $S$ lifting $x.$ By our assumption on fields of fractions, there exists some $n\in \N$ such that $y^{p^n}\in R.$ If we choose such an $n,$ we see that $x^{p^n}$ is contained in the residue field of $R.$ Hence the claim follows.\\
(ii) The claim that $S$ is finite over $R$ follows because $R$ is excellent and hence Japanese (see Proposition \ref{japaneseprop}). Therefore we know already that $S$ is a Dedekind domain all of whose prime ideals are principal. Hence all that remains to be shown is that $S$ is local. Let $\mathfrak{p},$ $\mathfrak{q}$ be two non-zero prime ideals in $S.$ Then $\mathfrak{p}\cap R$ and $\mathfrak{q}\cap R$ are both equal to the maximal ideal of $R.$ In particular, each element of $\mathfrak{p}$ has a power which lies in $\mathfrak{q},$ and vice-versa. This forces $\mathfrak{p}=\mathfrak{q}.$ 
\end{proof}\\
For later use, recall the following \it structure theorem for quasi-finite schemes over Henselian local rings:\rm
\begin{proposition}
Let $X\to \Spec R$ be a quasi-finite morphism of schemes, where $R$ is a Noetherian Henselian local ring. Then $X$ admits a unique decomposition $X=X^{\mathrm{f}}\sqcup X^{\eta}$ into disjoint open (and closed) subschemes such that $X^{\mathrm{f}}\to \Spec R$ is finite and such that the special fibre of $X^{\eta}\to \Spec R$ is empty. This decomposition is functorial in $X.$ \label{decompprop}
\end{proposition}
\begin{proof}
See \cite{Con2}, Theorem 4.10.  
\end{proof}\\
Now we let $\Og_K$ be as at the beginning of this Paragraph, and let $P$ be a pseudo-Abelian variety over $K.$ Choose a finite, purely inseparable extension $L$ of $K$ such that there is an exact sequence
$$0\to U\to P_L\to A\to 0,$$ where $U$ and $A$ are a smooth, connected, commutative unipotent group over $L$ and an Abelian variety over $L,$ respectively.

\begin{proposition}
Let $\Og_K$ and $P$ be as above. Further let $\mathscr{P}\to \Spec \Og_K$ be the Néron model of $P.$ Denote by $\Og_L$ the integral closure of $\Og_K$ in $L$ and by $\mathscr{A}$ the Néron model of $A$ over $\Og_L.$ Let $n\in \N$ and let $\ell$ be a prime number invertible in $\Og_K.$ Then the morphisms 
$$(\mathscr{P}_{\Og_L})[\ell^n]\to \mathscr{A}[\ell^n]$$ and $$\mathscr{P}[\ell^n]\to \Res_{\Og_L/\Og_K}\mathscr{A}[\ell^n]$$ induced by $\mathscr{P}_{\Og_L}\to \mathscr{A}$ and $\mathscr{P}\to \Res_{\Og_L/\Og_K}\mathscr{A}$, respectively, are isomorphisms. \label{torsisoprop}
\end{proposition}
\begin{proof}
Let us begin with the first isomorphism. Since $\mathscr{A}[\ell^n]$ is étale over $\Og_L$ and $\Spec \Og_L\to \Spec \Og_K$ is a universal homeomorphism by Lemma \ref{univlem}, there exists a unique étale group scheme $\mathscr{I}$ over $\Og_K$ together with an isomorphism 
$$\mathscr{I}\times_{\Og_K}\Spec \Og_L\to \mathscr{A}[\ell^n]$$ (see Proposition \ref{topolequivprop}). Now observe that the generic fibre of $\mathscr{I}$ is canonically isomorphic to $P[\ell^n],$ so the universal property of the Néron model gives us a morphism $\mathscr{I}\to \mathscr{P}$ extending this isomorphism. This map, in turn, induces a morphism
$$\mathscr{A}[\ell^n]=\mathscr{I}\times_{\Og_K}\Spec \Og_L\to (\mathscr{P}_{\Og_L})[\ell^n],$$ which is an inverse to the morphism from the Proposition. The second isomorphism can be constructed from the first by applying the functor $\Res_{\Og_L/\Og_K}-.$ 
\end{proof}
\begin{corollary}
Keep the notation from the previous Proposition and let $\kappa'$ be the residue field of $\Og_L.$ Then we have \\
(i) $\alpha(P)=\alpha(\Res_{L/K}A)$, \\
(ii) $\alpha(\mathscr{P}\times_{\Og_K}\Spec \kappa)=\alpha((\Res_{\Og_L/\Og_K}\mathscr{A})\times_{\Og_K}\Spec \kappa),$\\
(iii) $t(\mathscr{P}\times_{\Og_K}\Spec \kappa)=t((\Res_{\Og_L/\Og_K}\mathscr{A})\times_{\Og_K}\Spec \kappa),$\\
(iv) $\alpha(\mathscr{P}\times_{\Og_K}\Spec \kappa')=\alpha(\mathscr{A}\times_{\Og_L}\Spec \kappa'),$ and \\
(v) $t(\mathscr{P}\times_{\Og_K}\Spec \kappa')=t(\mathscr{A}\times_{\Og_L}\Spec \kappa').$
\end{corollary}
\begin{proof}
This follows from Proposition \ref{torsisoprop} together with Proposition \ref{alphainvariantprop}.
\end{proof}\\
Now fix a separable closure $K\sep$ of $K.$ Let $\Og_K^{\mathrm{sh}}$ be a strict Henselization of $\Og_K$ given by the choice of a separable closure $\kappa\sep$ of $\kappa$ and let $\Og_K\to \Og_K^{\mathrm{sh}}$ be the corresponding embedding. Let $\Og_{K,0}$ be the localization of the integral closure $\overline{\Og}_{K}$ of $\Og_K$ in $K\sep$ at a maximal ideal $\mathfrak{n}$ lying over the maximal ideal of $\Og_K^{\mathrm{sh}}.$ Such a choice gives us an embedding $\Og_K^{\mathrm{sh}}\to \Og_{K,0}$ (since $\Og_{K,0}$ is strictly Henselian local) and hence and embedding $K^{\mathrm{sh}}\to K\sep.$ Let $K^{\mathrm{sh}}:=\Frac \Og_K^{\mathrm{sh}}.$ We let the \it inertia group \rm $I_K$ of $K$ be the subgroup $\Gal(K\sep/K^{\mathrm{sh}})$ of $\Gal(K\sep/K).$ Of course, $I_K$ depends upon the choice of a maximal ideal as above, but this ambiguity is harmless for our purposes.

Now suppose that $L$ be a finite purely inseparable extension of $K,$ such as the one chosen at the beginning of this paragraph. Then $L\sep:=L\otimes_KK\sep$ is a separable closure of $L,$ and the absolute Galois groups $\Gal(L\sep/L)$ and $\Gal(K\sep/K)$ are canonically isomorphic as profinite groups via the morphism $\Gal(L\sep/L)\to \Gal(K\sep/K),$ $\sigma\mapsto \sigma\mid_{K\sep}.$ 

Denoting the residue field of $\Og_L$ by $\kappa',$ we observe that $\Og_L\otimes_{\Og_K}\Og_K^{\mathrm{sh}}$ is the strict Henselization of $\Og_L$ with respect to the separable closure $\kappa'\otimes_{\kappa}\kappa\sep$ of $\kappa'.$ The morphism $\Spec\overline{\Og}_{L}\to \Spec \overline{\Og}_{K}$ (where $\overline{\Og}_{L}$ is the integral closure of $\Og_L$ in $L\otimes_KK\sep$) is a homeomorphism, so there is a unique maximal ideal $\mathfrak{n'}\subseteq \overline{\Og}_L$  corresponding to $\mathfrak{n}.$ Let $\Og_{L,0}$ be the localization of $\overline{\Og}_{L}$ at $\mathfrak{n}'.$ We obtain a unique embedding
$$\Og_L\otimes_{\Og_K}\Og_K^{\mathrm{sh}}\to \Og_{L,0}$$ compatible with the embedding $\kappa'\otimes_{\kappa}\kappa\sep\to \Og_{L,0}/\mathfrak{n'}$ (note that the target of this morphism is an algebraically closed field), and hence an embedding $L\otimes_KK^{\mathrm{sh}}\to L\otimes_KK\sep$. The subgroups $\Gal(L\otimes_KK\sep/L\otimes_KK^{\mathrm{sh}})$ and $\Gal(K\sep/K^{\mathrm{sh}})$ are canonically identified under the isomorphism $\Gal(L\sep/L)\to \Gal(K\sep/K)$ constructed above. \

We are now ready to state and prove our first main result. The proof of the main new implications is self-contained, in so far as it does not use the well-known analogous result for Abelian varieties. We follow the usual proof for Abelian varieties (see, for example, \cite{BLR}, Chapter 7.4, proof of Theorem 5) quite closely.

\begin{theorem}
Using the same notation as before, the following are equivalent: \\
(i) The pseudo-Abelian variety $P$ has good reduction over $\Og_K,$ \\
(ii) the pseudo-Abelian variety $\Res_{L/K} A$ has good reduction over $\Og_K,$\\
(iii) the Abelian variety $A$ has good reduction over $\Og_L,$\\
(iv) there exists a prime number $\ell\in \Og_K^\times$ such that the Galois representation on $T_\ell(P)$ is unramified, and \\
(v) for all prime numbers $\ell\in \Og_K^\times,$ the Galois representation on $T_\ell(P)$ is unramified.  \label{NOScritthhm}
\end{theorem} 
\begin{proof}
Note that, for all claims made in the Proposition, the truth value does not change when replacing $\Og_K$ by a strict Henselization of $\Og_K$ (\cite{BLR}, Chapter 7.2, Corollary 2). Hence we shall assume that $\Og_K$ be strictly Henselian (so that, as a consequence, $\kappa$ is separably closed). In particular, we have $I_K=\Gal(K\sep/K).$ Let $\mathscr{P}$ be the Néron model of $P$ over $\Og_K.$ Then there is a finite étale group scheme $\Phi$ over $\kappa$ together with an exact sequence
$$0\to \mathscr{P}^0_{\kappa}\to \mathscr{P}_{\kappa}\to \Phi\to 0.$$
Furthermore, we can find a torus $T$, a smooth, connected unipotent group $U,$ and an Abelian variety $B$ over $\kappa\alg$ together with an exact sequence
$$0\to T\times_{\kappa\alg}U\to \mathscr{P}^0_{\kappa\alg}\to B\to 0$$
(see \cite{Con}, Theorem 1.1 and \cite{SGA3}, Exposé XVII, Théorème 7.2.1 b)).\\
$(i)\Rightarrow (v):$ Let $\ell$ be a prime number invertible in $\Og_K$ and $n\in \N.$ By Proposition \ref{decompprop}, $\mathscr{P}[\ell^n]$ can be written as a disjoint union 
$$\mathscr{P}[\ell^n]=(\mathscr{P}[\ell^n])^{\mathrm{f}}\sqcup(\mathscr{P}[\ell^n])^{\eta},$$ where $(\mathscr{P}[\ell^n])^{\mathrm{f}}$ is finite over $\Og_K$ and $(\mathscr{P}[\ell^n])^{\eta}$ has empty special fibre. This follows because $\Og_K$ is Henselian and $\mathscr{P}[\ell^n]$ is quasi-finite étale over $\Og_K.$ Because $P$ has good reduction, we must have 
$$\ord \mathscr{P}[\ell^n]_K\leq\ord \mathscr{P}[\ell^n]_\kappa;$$ the other inequality follows anyway because $\mathscr{P}[\ell^n]$ is quasi-finite étale over $\Og_K.$ This forces $(\mathscr{P}[\ell^n])^{\eta}$ to be empty, which implies that $\mathscr{P}[\ell^n]$ is finite over $\Og_K.$ Since $\Og_K$ is strictly Henselian, $\mathscr{P}[\ell^n]$ must (as a scheme) be a finite disjoint union of copies of $\Spec \Og_K,$ so the Galois action on $P[\ell^n](K\sep)$ is trivial. This triviality carries over to the limit. \\
$(v) \Rightarrow (iv)$ is obvious. \\
$(iv) \Rightarrow (i)$ Let $\ell$ be a prime number satisfying the condition of $(iv).$ By Proposition \ref{deltajumpprop}, we only have to exclude the case $\delta(P)>0.$ Using the fact that the map $\mathscr{P}_\kappa[\ell^n](\kappa)\to \mathscr{P}_\kappa[\ell^n](\kappa\alg)$ is an isomorphism, we can use the two exact sequences mentioned at the beginning of this proof to show that 
$$\ell^{n\cdot 2\alpha(P)}\leq \ord \mathscr{P}_\kappa[\ell^n](\kappa)\leq \ord \Phi \cdot \ell^{n(t(\mathscr{P}_\kappa)+2\alpha(\mathscr{P}_\kappa))}.$$ The first inequality follows from the fact that the Galois representation on $P[\ell^n](K\sep)$ is trivial by our assumption in $(iv)$, which means that $P[\ell^n]$ is a constant group scheme over $K.$ Therefore it has a finite étale model over $\Og_K,$ which admits a closed embedding into $\mathscr{P}$ by the universal property of the Néron model. The inequality above can be rearranged as
$$1\leq \ord\Phi\cdot \ell^{n(t(\mathscr{P}_\kappa)-2\delta(P))}.$$ In order for this to be valid for all $n,$ we must have $t(\mathscr{P}_{\kappa})\geq 2\delta(P).$ Since we also have $t(\mathscr{P}_{\kappa})\leq \delta(P)$ by Proposition \ref{deltajumpprop}, this forces $\delta(P)=0.$ \\
Because the Galois representations $T_{\ell}(P),$ $T_{\ell}(A),$ and $T_{\ell}(\Res_{L/K}A)$ are canonically identified, we can use the same arguments to show the implications $(ii)\Rightarrow (v)\Rightarrow(iv)\Rightarrow(ii)$ and $(iii)\Rightarrow (v)\Rightarrow(iv)\Rightarrow(iii).$ This concludes the proof. 
\end{proof}\\
$\mathbf{Remark.}$ In analogy with the case of Abelian varieties, one might have guessed that the correct definition for \it good reduction \rm of pseudo-Abelian varieties should be the requirement that the identity component of the special fibre of the Néron model be itself pseudo-Abelian. There are two reasons why this is not the case. Firstly, with this definition, the analogue of the Néron-Ogg-Shafarevich criterion we proved above would not hold, and secondly, this alternative definition would not be equivalent to the usual definition for Abelian varieties. We shall now give an example exhibiting both of those phenomena: Let $R$ be an excellent strictly Henselian discrete valuation ring with non-perfect residue field $\kappa$ and uniformizer $\pi.$ Let $a\in R$ be an element whose image $\overline{a}$ in the residue field $\kappa$ of $R$ is not a $p$-th power, where $p=\mathrm{char}\, \kappa=\mathrm{char}\, R.$ Let $K$ be the field of fractions of $R$ and let $L:=K[X]/\langle X^p+\pi X-a\rangle.$ Then the extension $L/K$ is separable. However, if we let $S$ be the integral closure of $R$ in $L,$ (which is a discrete valuation ring because $R$ is Henselian), then the induced extension $\kappa\subseteq \kappa'$ of residue fields is purely inseparable. Indeed, we must have $[\kappa':\kappa]\leq [L:K]=p,$ but the image of $X$ in $\kappa'$ has degree $p$ over $\kappa,$ so this inequality is an equality and $\kappa'=\kappa(\overline{a}^{1/p})$. In particular, $S\otimes_R\kappa=\kappa'.$ Now let $\mathscr{E}\to \Spec S$ be an Abelian scheme over $S.$ Then $\mathscr{E}$ is the Néron model of its generic fibre, and hence so is $\mathscr{E}':=\Res_{S/R}\mathscr{E}$ (at this point we do have to use that $\mathscr{E}\to \Spec S$ is projective; see \cite{BLR}, Chapter 6.4, Theorem 1). Because $L/K$ is separable, the generic fibre of $\mathscr{E}'$ is proper. However, since the residue field extension is inseparable, the special fibre of $\mathscr{E}'$ is not proper (see \cite{CGP}, Example A.5.6). The special fibre is still a pseudo-Abelian variety since it is clearly smooth and connected (for the latter claim, see \cite{CGP}, Corollary A.5.9), and if $G$ is a smooth connected affine algebraic group over $\kappa,$ a morphism $G\to \mathscr{E}'_{\kappa}$ is the same a morphism $G_{\kappa'}\to \mathscr{E}_{\kappa'},$ and clearly all such morphisms vanish. \\
\\
If $P$ is a pseudo-Abelian variety over $K,$ there exists a unique exact sequence
$$0\to E\to P\to V\to 0,$$ where $E$ and $V$ denote an Abelian variety and a smooth connected unipotent group over $K,$ respectively. This follows from \cite{T}, Theorem 2.10. We have
\begin{proposition}
With the notation above, the dimension of $E$ is equal to the virtual Abelian rank of $P.$ Furthermore, $P$ has good reduction over $\Og_K$ if and only if so does $E.$ \label{Eprop}
\end{proposition}
\begin{proof}
Replace $K$ by a finite purely inseparable extension such that there exists an exact sequence $0\to U \to P\to A\to 0$ with $A$ an Abelian variety and $U$ a smooth connected unipotent group over $K.$ Choose a prime number $\ell$ invertible in $K.$ Then the maps $T_\ell(E)\to T_\ell(P)\to T_\ell(A)$ are isomorphisms of Galois representations, which proves both claims. 
\end{proof}

\subsection{Properties of pseudo-Abelian varieties with good reduction}
We shall see that the Néron model of a pseudo-Abelian variety $P$ behaves like the Néron model of an Abelian variety with good reduction in some important ways, but the behaviour can be very different in some other respects. As before, we let $\Og_K$ be an excellent discrete valuation ring with residue field $\kappa$ and field of fractions $K,$ both assumed to be of characteristic $p>0.$ Further, we let $P$ be a pseudo-Abelian variety over $K$. 

\begin{proposition}
Let $P$ be a pseudo-Abelian variety over $K$ with good reduction. Let $F/K$ be a finite separable extension. Let $S$ be the localization of the integral closure of $\Og_K$ in $F$ at a non-zero prime ideal. Then $P_F$ has good reduction over $S.$  \label{goodredpreseredprop}
\end{proposition} 
\begin{proof}
We may assume without loss of generality that $\Og_K$ be strictly Henselian (\cite{BLR}, Chapter 7.2, Corollary 2). Then the Galois representation on $T_{\ell}(P)$ is trivial for all prime numbers $\ell\in \Og_K^\times$. The Galois representation $T_{\ell}(P_F)$ arises from that of $T_{\ell}(P)$ by restricting the action of $\Gal(K\sep/K)$ to the subgroup $\Gal(K\sep/F),$ which is therefore trivial as well. Hence the claim follows from Theorem \ref{NOScritthhm}, $(v)\Rightarrow (i).$ 
\end{proof}\\
$\mathbf{Remark.}$ In is not true in general that Néron models of pseudo-Abelian varieties with good reduction commute with faithfully flat base change (although this is always the case for Abelian varieties). An example can be constructed as follows: Let $\Og_K$ be a complete discrete valuation ring with algebraically closed residue field. Let $L/K$ be a finite non-trivial purely inseparable extension, and let $F$ be a finite non-trivial separable extension. Let $\mathscr{E}\to \Spec \Og_K$ be an elliptic curve with generic fibre $E.$ We can choose an isomorphism 
$$\psi\colon \Og_K\to \Lie \mathscr{E}$$ (see paragraph 1 of \cite{LLR} for an introduction to Lie algebras). Let $P:=\Res_{L/K}E_L.$ Then the Néron model $\mathscr{P}$ of $P$ is isomorphic to $\Res_{\Og_L/\Og_K} \mathscr{E}_{\Og_L}$ in a canonical way. Hence the Lie algebra of $\mathscr{P}$ is equal to ($\Lie \mathscr{E})\otimes_{\Og_K}\Og_L$ (viewed as a module over $\Og_K$). The Lie algebra of the Néron model $\mathscr{P}_F$ of $P_F$ is equal to $(\Lie \mathscr{E})\otimes_{\Og_K}\Og_{F\otimes_KL}$ (viewed as a module over $\Og_F$) by the same argument. We obtain a commutative diagram
$$\begin{CD}
(\Lie \mathscr{P})\otimes _{\Og_K}\Og_F@>>>\Lie \mathscr{P}_F\\
@VVV@VVV\\
(\Lie \mathscr{E})\otimes_{\Og_K}\Og_L\otimes_{\Og_K}\Og_F@>>>(\Lie \mathscr{E})\otimes_{\Og_K}\Og_{L\otimes_KF}\\
@V{\psi\otimes \mathrm{Id}}VV@VV{\psi\otimes\mathrm{Id}}V\\
\Og_L\otimes_{\Og_K}\Og_F@>>>\Og_{L\otimes_KF}.
\end{CD}$$
The vertical maps are all isomorphisms. However, the bottom horizontal map is not an isomorphism (indeed, choose uniformizers $\pi_F$ and $\pi_{L\otimes_KF}$ of $\Og_F$ and $\Og_{L\otimes_KF},$ respectively, and $\epsilon\in \Og_{L\otimes_KF}^\times$ such that $\pi_F=\epsilon\pi_{L\otimes_KF}^{[L:K]}.$ Then $\Og_L\otimes_{\Og_K}\Og_F=\Og_L[\pi_F]=\Og_L[\epsilon\pi_{L\otimes_KF}^{[L:K]}]\subsetneq \Og_L[\epsilon\pi_{L\otimes_KF}]=\Og_{L\otimes_KF}$; see \cite{Serre}, Chapitre 1, Proposition 18).\\
\\
If $A$ is an Abelian variety over $K,$ then a smooth separated model $\mathscr{A}\to \Spec\Og_K$ with the property that $\delta(\mathscr{A})=0$ is unique up to unique isomorphism, if it exists. This follows because the condition stated above implies that $\mathscr{A}\to \Spec \Og_K$ is proper, and hence the Néron model of its generic fibre. The remark above shows that this fails for pseudo-Abelian varieties. However, something only marginally weaker is true:
\begin{proposition}
Let $\Og_K$ be an excellent discrete valuation ring with field of fractions $K.$ Let $P$ be a pseudo-Abelian variety over $K$. Suppose that there exist a smooth separated model $\mathscr{P}'$ over $\Og_K$ of $P$ of finite type such that $\delta(\mathscr{P}')=0.$ Then $P$ has good reduction. 
\end{proposition}
\begin{proof}
We may assume without loss of generality that $\Og_K$ be strictly Henselian (\cite{BLR}, Chapter 7.2, Corollary 2). It suffices to show that, for some prime number $\ell$ invertible in $\Og_K,$ the finite étale $K$-group schemes $P[\ell^n]$ admit finite étale models over $\Og_K.$ Indeed, this will imply that the group schemes $P[\ell^n]$ are constant, so the Galois action on $T_\ell(P)$ is trivial, which implies our result by Theorem \ref{NOScritthhm}. Hence it is enough to prove that the $\Og_K$-group schemes $\mathscr{P}'[\ell^n]$ are finite over $\Og_K$ for all $n\geq 0.$ The schemes $\mathscr{P}'[\ell^n]$ are clearly quasi-finite étale over $\Og_K,$ so we have a decomposition 
$$\mathscr{P}'[\ell^n]=\mathscr{P}'[\ell^n]^{\mathrm{f}}\sqcup \mathscr{P}'[\ell^n]^{\eta},$$ where $\mathscr{P}'[\ell^n]^{\eta}$ has empty special fibre and $\mathscr{P}'[\ell^n]^{\mathrm{f}}$ is finite over $\Og_K$ (see Proposition \ref{decompprop}). Our assumption on the virtual Abelian ranks implies that $\mathscr{P}'[\ell^n]^{\eta}=\emptyset,$ so that our claim follows. 
\end{proof}
\subsection{The group of connected components}
For an Abelian variety $A$ over a discretely valued field $K$ with separably closed residue field $\kappa$, Grothendieck (\cite{SGA7 IX}, Paragraph 11) proved that the formula 
$$\Phi(\ell)(\kappa)=H^1(I,T_{\ell}(A))_{\mathrm{tors}},$$ where $I$ denotes an inertia group of $K$ and $\ell$ a prime number invertible in $\kappa.$ The cohomology refers to continuous Galois cohomology. We shall briefly recall the argument in order to show that it holds for pseudo-Abelian varieties $P\to \Spec K.$ First note that we may assume without loss of generality that $\Og_K$ be strictly Henselian, so that $I:=I_K=\Gal(K\sep/K).$ Suppose $\mathscr{P}$ denotes the Néron model of $P$ over $\Og_K.$ We have 
$$T_{\ell}(P)^{I}=\varprojlim \mathscr{P}^0[\ell^n](\Og_K).$$ For each $n>0,$ we have an exact sequence
$$0\to \mathscr{P}^0[\ell^n](\Og_K)\to \mathscr{P}[\ell^n](\Og_K)\to \Phi[\ell^n](\kappa)\to 0,$$ which is the same as an exact sequence
$0\to T_\ell(P)^{I}\otimes_{\Z_\ell}\Z/\ell^n\Z \to (T_\ell(P)\otimes_{\Z_\ell}\Z/\ell^n\Z)^{I}\to \Phi[\ell^n](\kappa)\to 0.$ Taking inductive limits, we obtain
an exact sequence $$0\to T_{\ell}(P)^{I}\otimes_{\Z_{\ell}} \Q_\ell/\Z_\ell\to (T_{\ell}(P)\otimes_{\Z_\ell}\Q_\ell/\Z_{\ell})^{I}\to \Phi(\ell)(\kappa)\to 0.$$ Now consider the exact sequence of $I$-representations
$$0\to T_\ell(P)\to T_\ell(P)\otimes_{\Z\ell}\Q_\ell\to T_\ell(P)\otimes_{\Z_\ell}\Q_\ell/\Z_\ell\to 0.$$ By considering the long exact cohomology sequence, we can construct a canonical exact sequence 
$$0\to \phi(\ell)(\kappa)\to H^1(I, T_\ell(P))\to H^1(I, T_\ell(P)\otimes_{\Z_\ell} \Q_\ell)=H^1(I, T_\ell(P))\otimes_{\Z_\ell} \Q_\ell,$$ which implies Grothendieck's formula. Now suppose only that $\Og_K$ be excellent (without assuming that $\kappa$ be separably closed).
By \cite{T}, Theorem 2.1, there is a unique exact sequence 
$$0\to E\to P\to V\to 0,$$ where $E$ is an Abelian variety and $V$ a smooth connected commutative unipotent algebraic group over $K,$ respectively. Because $V[\ell^n]=0$ for all $n\geq 0,$ we see immediately that the induced morphism $T_\ell(E)\to T_\ell(P)$ is an isomorphism. Hence we obtain
\begin{proposition}
Let $P$ be a pseudo-Abelian variety over $K$ and let $E$ be the maximal Abelian subvariety of $P$ as above. Let $\Phi_E$ and $\Phi_P$ be the group schemes of connected components of the Néron models of $E$ and $P,$ respectively. Then, for any prime number $\ell$ invertible in $\Og_K,$ the canonical map 
$$\Phi_E(\ell)\to \Phi_P(\ell)$$ is an isomorphism. 
\end{proposition}
\begin{proof}
First assume that $\Og_K$ be strictly Henselian. Then we have a commutative diagram
$$\begin{CD}
\Phi_E(\ell)(\kappa)@>>>\Phi_P(\ell)(\kappa)\\
@VVV@VVV\\
H^1(I, T_\ell(E))_{\mathrm{tors}}@>>>H^1(I, T_\ell(P))_{\mathrm{tors}},
\end{CD}$$
where the horizontal maps are the obvious ones and the vertical maps are those constructed in the discussion above. The vertical arrows are isomorphisms by construction and the bottom horizontal arrow is an isomorphism because so is $T_\ell(E)\to T_\ell(P).$ Hence the top horizontal arrow is an isomorphism. If $\Og_K$ is not necessarily strictly Henselian, we still obtain an isomorphism $\Phi_E(\kappa\sep)\to \Phi_P(\kappa\sep)$ as above, and the naturality of the Néron model shows that this morphism is Galois equivariant. This means that the isomorphism constructed above descends to an isomorphism of étale $\kappa$-group schemes $\Phi_E(\ell)\to \Phi_P(\ell).$
\end{proof}\\
As a consequence, we obtain the following 
\begin{corollary}
Let $P$ be a pseudo-Abelian variety over $K$ with good reduction. Then $\ord \Phi_P$ is a power of $p=\mathrm{char}\, \kappa=\mathrm{char}\, K.$ 
\end{corollary}
\begin{proof}
Write $P$ as an extension of a smooth connected commutative unipotent group $U$ by an Abelian variety $E$ as above. Since the map $T_\ell(E)\to T_\ell(P)$ is an isomorphism, Theorem \ref{NOScritthhm} implies that $E$ has good reduction. Since $E$ is an Abelian variety, this means that $\Phi_E=0.$ By the Proposition above, $\Phi_P(\ell)=0$ for all $\ell\not=p.$
\end{proof}\\
It is not known whether there exist pseudo-Abelian varieties $P$ over $K$ with good reduction such that $\Phi_P\not=0.$ We shall consider some examples of pseudo-Abelian varieties below; in each case, we shall see that the component group is trivial. Let us begin with 
\begin{proposition}
Let $P$ be a pseudo-Abelian variety over the field $K$ which is isomorphic to $\Res_{L/K} A$ for some Abelian variety $A$ over a finite purely inseparable extension $L$ of $K$ which has good reduction over the integral closure $\Og_L$ of $\Og_K$ in $L.$  Then $\Phi_P=0.$ 
\end{proposition}
\begin{proof}
Let $\mathscr{A}$ be the Néron model of $A$ over $\Og_L$ Then $\Res_{\Og_L/\Og_K}\mathscr{A}$ (which is representable by Proposition \ref{universalhomeoprop} and Lemma \ref{univlem}) is the Néron model of $P$ over $\Og_K,$ as can be seen easily by considering the universal property. Hence the special fibre of the Néron model of $P$ is isomorphic to $\Res_{\Og_L\otimes_{\Og_K}\kappa/\kappa} (\mathscr{A}\times_{\Og_L}\Spec {\Og_L}\otimes_{\Og_K}\kappa),$ which is connected by \cite{CGP}, Proposition A.5.9.
\end{proof}\\
Now recall the construction of another class of pseudo-Abelian varieties over $K$ given in \cite{T}, Lemma 8.1: Let $L$ be a purely inseparable extension of $K $ of degree $p$ with ring of integers $\Og_L$, let $U:=\Res_{L/K}\Gm/\Gm$ (which is a smooth connected commutative algebraic group over $K$ of exponent $p$), and let $E$ be an elliptic curve with the property that $E[p]\cong {\boldsymbol{\mu}_{p}}\times \Z/p\Z$ and such that $E$ can be defined over $K^p.$ Note that $R:=\Res_{L/K}\Gm$ is pseudo-reductive, hence \cite{T}, Lemma 8.1 applies.
If $\mathscr{G}_{\mathrm{m}}$ and $\mathscr{R}$ denote the Néron lft-models of $\Gm$ and $R,$ respectively, one convinces oneself easily that $\mathscr{N}:=\Res_{\Og_L/\Og_K}\mathscr{G}_{\mathrm{m}}/\mathscr{G}_{\mathrm{m}}$ is the Néron model of $U$ over $\Og_K$ (indeed, $\mathscr{N}$ is clearly smooth and of finite type, so the claim follows from Hilbert's theorem 90 together with \cite{BLR}, Chapter 7.1, Theorem 1). We also see that $\Phi_U\cong \Z/e_{L/K}\Z$ as a group scheme, where $e_{L/K}$ denotes the ramification index of the extension $L/K.$ This number is an element of the set $\{1,p\}.$ 
Recall that the extension $0\to \Gm\to R\to U\to 0$ comes from an extension $0\to {\boldsymbol{\mu}_{p}}\to H\to U\to 0$ in a unique way, and the push-out $P$ of the maps ${\boldsymbol{\mu}_{p}}\to H$ and ${\boldsymbol{\mu}_{p}}\to E$ is a pseudo-Abelian variety over $K$ (see \cite{T}, Lemma 8.1 for both those claims). The remainder of this paragraph will be dedicated to proving the following
\begin{proposition}
Suppose $P$ arise from the construction recalled above. Suppose further that $E$ have good reduction over $\Og_K,$ and that the Néron model $\mathscr{E}$ of $E$ have the property that there exists an isomorphism 
$$\mathscr{E}[p]\cong {\boldsymbol{\mu}_{p}}\times\Z/p\Z$$ over $\Og_K.$ Then the Néron model $\mathscr{P}$ of $P$ over $\Og_K$ has the following properties: \\
(i) The canonical map $\mathscr{E}\to \mathscr{P}$ is a closed immersion, and \\
(ii) The morphism $\mathscr{P}\to \mathscr{N}$ factors through $\mathscr{N}^0$ and the induced map $\mathscr{P}/\mathscr{E}\to \mathscr{N}^0$ is an isomorphism. In particular, we have $\Phi_P=0.$ \label{exampleprop}
\end{proposition}
\begin{proof}
Let us begin by showing (i). Observe that the sequence $0\to \Gm\to R\to U\to 0$ splits over $L.$ Hence the same is true for the extension $0\to E\to P\to U\to 0.$ Therefore, the morphism $\mathscr{E}\to\mathscr{P}$ acquires a retraction after base change to $\Og_L$ (this follows from the universal property of the Néron model because $E$ has good reduction). Hence $\mathscr{E}\to \mathscr{P}$ becomes a closed immersion after an fppf-cover and is therefore a closed immersion itself. 

Now we shall prove that the extension $0\to {\boldsymbol{\mu}_{p}}\to H\to U\to 0$ extends canonically to an extension $0\to {\boldsymbol{\mu}_{p}}\to \mathscr{H}\to \mathscr{N}^0\to 0.$ Consider the element $[0\to \Gm\to \mathscr{R}^0\to \mathscr{N}^0\to 0]\in \Ext^1(\mathscr{N}^0, \Gm).$ We have an exact sequence 
$$0=\Hom(\mathscr{N}^0, \Gm)\to \Ext^1(\mathscr{N}^0,{\boldsymbol{\mu}_{p}})\to \Ext^1(\mathscr{N}^0, \Gm)\overset{\cdot p}{\to} \Ext^1(\mathscr{N}^0, \Gm);$$ the $\Ext$-groups are taken in the category of fppf-sheaves. The last map in this sequence is equal to zero since $\mathscr{N}^0$ is killed by $p.$ Hence $\Ext^1(\mathscr{N}^0, {\boldsymbol{\mu}_{p}})\to \Ext^1(\mathscr{N}^0, \Gm)$ is an isomorphism and the element $[0\to \Gm\to \mathscr{R}^0\to \mathscr{N}^0\to 0]$ of $\Ext^1(\mathscr{N}^0, \Gm)$ comes uniquely from an element of $\Ext^1(\mathscr{N}^0, {\boldsymbol{\mu}_{p}})$. This element is represented by an exact sequence $0\to {\boldsymbol{\mu}_{p}}\to \mathscr{F}\to \mathscr{N}^0\to 0$ for some fppf-sheaf $\mathscr{F}.$ This sheaf is clearly a separated algebraic space\footnote{This is because fppf-descent of algebraic spaces is always effective; see, for example, the Stacks Project \cite{Stacks}, Tag 0ADV.} of finite presentation over $\Og_K$ with a group structure, so it is representable by \cite{An}, Chapitre IV, Théorème 4.B. This also shows that the extension $0\to E\to P\to U\to 0$ extends uniquely to an exact sequence 
$$0\to \mathscr{E}\to \widehat{\mathscr{P}}\to \mathscr{N}^0\to 0.$$
Our goal is to show that the canonical map $\widehat{\mathscr{P}}\to \mathscr{P}$ is an isomorphism, which will clearly imply claim (ii) from the Proposition. Because $\mathscr{E}\to \mathscr{P}$ is a closed immersion, there exists a canonical map $\mathscr{P}/\mathscr{E}\to \mathscr{N},$ and the map $\widehat{\mathscr{P}}\to \mathscr{P}$ induces a morphism $\mathscr{N}^0=\widehat{\mathscr{P}}/\mathscr{E}\to \mathscr{P}/\mathscr{E}.$ We shall need
\begin{lemma}
The map $\mathscr{P}/\mathscr{E}\to \mathscr{N}$ is an open immersion.
\end{lemma}
\begin{proof}
First observe that the map above is étale. Using the fibre-wise criterion of flatness, we can check this at the two fibres separately. The claim for the generic fibre is obvious. Since the composition $\mathscr{N}^0\to \mathscr{P}/\mathscr{E}\to \mathscr{N}$ is the canonical open immersion, the induced map on special fibres is étale as well, so the claim follows. Hence the kernel of $\mathscr{P}/\mathscr{E}\to \mathscr{N}$ is quasi-finite étale, and we already know that it is trivial generically. By looking at the base change of this kernel to $\Og_K^{\mathrm{sh}},$ we see that it must be finite over $\Og_K,$ and hence trivial. Putting things together, we see that $\mathscr{P}/\mathscr{E}\to \mathscr{N}$ is an étale monomorphism of schemes, and hence an open immersion. 
\end{proof}\\
We must now distinguish two cases. Assume first that $e_{L/K}=1.$ Then $\mathscr{N}=\mathscr{N}^0,$ and the lemma above shows that the sequence $0\to\mathscr{E}\to \mathscr{P}\to \mathscr{N}\to 0$ is exact. Hence $\widehat{\mathscr{P}}=\mathscr{P}$ in this case. \\
Now suppose that $e_{L/K}=p.$ Then the map $\mathscr{P}/{\mathscr{E}}\to \mathscr{N}$ is either surjective or induces an isomorphism $\mathscr{P}/{\mathscr{E}}\to \mathscr{N}^0.$
Hence we only have to exclude the first case. Suppose therefore, in order to derive a contradiction, that the sequence 
\begin{align}0\to \mathscr{E}\to \mathscr{P}\to \mathscr{N}\to 0\label{extn}\end{align}
 is exact. Considering the exact sequence $0\to {\boldsymbol{\mu}_{p}}\to \mathscr{E}\to \mathscr{E}/{\boldsymbol{\mu}_{p}}\to 0,$ we obtain an exact sequence
$$0=\Hom(\mathscr{N}, \mathscr{E}/{\boldsymbol{\mu}_{p}})\to \Ext^1(\mathscr{N}, {\boldsymbol{\mu}_{p}})\to \Ext^1(\mathscr{N}, \mathscr{E})\to \Ext^1(\mathscr{N}, \mathscr{E}/{\boldsymbol{\mu}_{p}}).$$
We claim that the image of $(\ref{extn})$ in $\Ext^1(\mathscr{N}, \mathscr{E}/{\boldsymbol{\mu}_{p}})$ vanishes. This is equivalent to the claim that the sequence
$$0\to \mathscr{E}/{\boldsymbol{\mu}_{p}}\to \mathscr{P}/{\boldsymbol{\mu}_{p}}\to \mathscr{N}\to 0$$ splits. However, the second and fourth term of this sequence are Néron models (of finite type) of their respective generic fibres, which implies that so is the middle term (see \cite{BLR},  Chapter 7.5, proof of Proposition 1(b)). It follows from the construction of $P$ that the sequence $0\to E/{\boldsymbol{\mu}_{p}}\to P/{\boldsymbol{\mu}_{p}}\to U\to 0$ splits, and this splitting induces one at the level of Néron models. Hence the exact sequence of $\Ext$-groups above shows that $(\ref{extn})$ comes from a unique exact sequence
$$0\to {\boldsymbol{\mu}_{p}}\to \mathscr{S}\to\mathscr{N}\to 0,$$ whose generic fibre must, by uniqueness, coincide with the extension of the same kind we used to construct $P.$ Taking the push-out of this sequence along ${\boldsymbol{\mu}_{p}}\to \mathscr{G}_{\mathrm{m}}$ gives us an exact sequence
$$0\to \mathscr{G}_{\mathrm{m}}\to \widehat{\mathscr{S}}\to \mathscr{N}\to 0.$$ However, this contradicts the following
\begin{lemma}
(i) The canonical morphism $\widehat{\mathscr{S}}\to \mathscr{R}$ is an isomorphism, \\
(ii) The sequence $0\to \mathscr{G}_{\mathrm{m}}\to \mathscr{R}\to \mathscr{N}\to 0$ does not come from an exact sequence $0\to {\boldsymbol{\mu}_{p}}\to \mathscr{S}\to \mathscr{N}\to 0.$ 
\end{lemma}
\begin{proof}
For part (i), all we have to show is that for all discrete valuation rings $T$ of ramification index 1 over $\Og_K,$ the canonical map $\widehat{\mathscr{S}}(T)\to \widehat{\mathscr{S}}(\Frac T)$ is bijective (\cite{BLR}, Chapter 10.1, Proposition 2). The same argument as in the proof of \cite{BLR}, Chapter 7.5, Proposition 1 reduces this claim to the assertion that the map $\widehat{\mathscr{S}}(T)\to \mathscr{N}(T)$ is surjective. This will follow if we can show that $H^1(\Spec T, \mathscr{G}_{\mathrm{m}})=0.$ However, if $j$ denotes the inclusion of the special point of $\Spec T,$ we have an exact sequence
$0\to \Gm\to \mathscr{G}_{\mathrm{m}}\to j_\ast\Z\to 0,$ which gives rise to an exact sequence $H^1(\Spec T, \Gm)\to H^1(\Spec T, \mathscr{G}_{\mathrm{m}})\to H^1(\Spec T, j_\ast \Z).$ Clearly, $H^1(\Spec T, \Gm)=\Pic \Spec T=0,$ and the last term of the sequence is equal to the Galois cohomology of the residue field of $T$ with coefficients in $\Z,$ which is trivial as well. \\
For part (ii), note that if $0\to \mathscr{G}_{\mathrm{m}}\to \mathscr{R}\to \mathscr{N}\to 0$ came from an exact sequence $0\to {\boldsymbol{\mu}_{p}}\to \mathscr{S}\to \mathscr{N}\to 0,$ then the sequence $0\to \mathscr{G}_{\mathrm{m}}/{\boldsymbol{\mu}_{p}}\to \mathscr{R}/{\boldsymbol{\mu}_{p}}\to \mathscr{N}\to 0$ would have to split, which is not the case because the component group of $\mathscr{R}/{\boldsymbol{\mu}_{p}}$ is isomorphic to $\Z,$ whereas the component group of $\mathscr{N}$ is isomorphic to $\Z/p\Z$ since we assume $e_{L/K}=p.$
\end{proof}\\
This Lemma finishes the proof of Proposition \ref{exampleprop}
\end{proof}\\
Both examples treated above seem to suggest that, for a pseudo-Abelian variety $P$ over $K$ with good reduction over $\Og_K$, it should be reasonable to expect that $\Phi_P=0.$ There is some further evidence that, for a pseudo-Abelian variety $P$ over $K$ (not necessarily with good reduction), the component group scheme $\Phi_P$ should vanish \it almost always \rm in the following sense: Suppose $S$ is an excellent Dedekind scheme with field of fractions $K,$ and let $P$ be a pseudo-Abelian variety over $K.$ If $P$ admits a Néron model $\mathscr{P}\to S,$ then the component groups will vanish at all but finitely many closed points of $T$ (see \cite{BLR}, Chapter 10.1, Corollary 10). Although it is not known whether pseudo-Abelian varieties admit Néron models over general Dedekind schemes, this would follow from resolution of singularities in characteristic $p.$ Indeed, one sees easily that the maximal unirational subgroup $\mathrm{uni}_K(P)$ of $P$ over $K$ is trivial (since such groups are always smooth, connected, and affine), so the existence of Néron models over general Dedekind schemes would follow from the existence of regular compactifications of pseudo-Abelian varieties (see \cite{BLR}, Chapter 10.3, Theorem 5(a)). Since this is widely believed to hold, our observations can be viewed as further evidence that almost all pseudo-Abelian varieties over discretely valued fields should have trivial component group. This motivates the following 
\begin{question}
Let $\Og_K$ be an excellent discrete valuation ring with field of fractions $K$ and let $P$ be a pseudo-Abelian variety over $K.$ Is it true that, if $P$ has good reduction, then $\Phi_P=0$?  \label{Phiquestion}
\end{question}
\section{Pseudo-semiabelian reduction}
In this paragraph, we shall define an analogue of semi-Abelian reduction for pseudo-Abelian varieties. The setup will be the same as in the last paragraph: We let $\Og_K$ be an excellent discrete valuation ring with residue field $\kappa$ and field of fractions $K,$ assumed to be of characteristic $p>0.$ Let $P$ be a pseudo-Abelian variety over $K.$ 
Classically, an Abelian variety $A$ over $K$ is said to have \it semiabelian reduction \rm if the identity component $\mathscr{A}^0$ of the Néron model $\mathscr{A}$ of $A$ is a semiabelian scheme over $\Og_K,$ i.e., if its special fibre is an extension of an Abelian variety by a torus. This is equivalent to the condition that the defect $\delta(A)$ be equal to the toric rank of the special fibre of $\mathscr{A}.$ If $P$ is a pseudo-Abelian variety over $K$ with Néron model $\mathscr{P}\to \Spec \Og_K,$ then we still have $t(\mathscr{P}_{\kappa})\leq \delta(P)$ by Proposition \ref{deltajumpprop}. 
\begin{definition}
The pseudo-Abelian variety $P$ over $K$ has \rm pseudo-semiabelian reduction \it (over $\Og_K$) if \label{pseudosemabdef}
$$\delta(P)=t(\mathscr{P}_{\kappa}).$$ 
\end{definition}
In the realm of classical semiabelian reduction, there are two fundamental results. The first is the \it semiabelian reduction theorem, \rm due originally to Grothendieck, which says that we can find a finite separable extension $F/K$ such that the Abelian variety acquires semiabelian reduction over the integral closure of $\Og_K$ in $F.$ The second is a characterization of semiabelian reduction in terms of the Galois representation on the Tate module of the Abelian variety. We shall see that both those results hold true in the world of pseudo-Abelian varieties as well. However, we shall also see that Néron models of pseudo-Abelian varieties with pseudo-semiabelian reduction behave quite differently from semiabelian schemes in some ways. 
\begin{theorem}
Let $P$ be a pseudo-Abelian variety over $K.$ Then there exists a finite separable extension $F/K$ with the following property: For all localizations $S$ of the integral closure of $\Og_K$ in $F$ at a non-zero prime ideal, $P_F$ has pseudo-semiabelian reduction over $S.$ 
\end{theorem}
\begin{proof}
Let $L$ be a finite purely inseparable extension of $K$ over which there is an exact sequence $0\to U\to P_L\to A\to 0,$ where $U$ is smooth, connected, commutative, and unipotent, and $A$ is an Abelian variety. By Grothendieck's theorem on semiabelian reduction (see \cite{BLR}, Chapter 7.4, Theorem 1), there exists a finite separable extension $F'$ of $L$ such that $A$ acquires semiabelian reduction over $F'$. Because $L/K$ is purely inseparable, $F'$ is of the form $F'=F\otimes_KL$ for some finite separable extension $F/K.$ We claim that $P$ acquires pseudo-semiabelian reduction over $F.$ Indeed, let $S$ be the localization of the integral closure of $\Og_K$ in $F$ at a non-zero prime, and let $S'$ be the integral closure of $S$ in $F\otimes_KL.$ Let $\mathscr{P}\to \Spec S$ be the Néron model of $P_F,$ and let $\mathscr{A}$ be the Néron model of $A_{F\otimes_KL}$ over $S'.$ Then the morphism 
$$\mathscr{P}\times_S \Spec S'\to \mathscr{A}$$
induces isomorphisms on $\ell$-torsion subschemes for all prime numbers $\ell\in \Og_K^\times$ by Proposition \ref{torsisoprop}, which implies (using Proposition \ref{alphainvariantprop}) that the invariants $\alpha(-)$ and $t(-)$ of the special fibres of $\mathscr{P}$ and $\mathscr{A}$ coincide. Because $A_{F\otimes_KL}$ has semiabelian reduction by our choice of $F,$ the claim of the Proposition follows.  
\end{proof}\\
In the world of Abelian varieties, semiabelian models satisfy a uniqueness property almost as strong as smooth proper models: Indeed, suppose $A$ is an Abelian variety over $K$ and that $\mathscr{A}\to \Spec \Og_K$ is a smooth separated model with \it connected special fibre. \rm  Assume further that $t(\mathscr{A}_{\kappa})=\delta(\mathscr{A}).$ Then $\mathscr{A}$ is isomorphic to the identity component of the Néron model of $A$. In particular, there is only one model of $A$ with those properties up to unique isomorphism. This fails for pseudo-Abelian varieties as shown by the remark after Proposition \ref{goodredpreseredprop}. We do, however, have the following
\begin{proposition}
Let $P$ be a pseudo-Abelian variety over $K,$ and suppose that $\mathscr{P}'\to \Spec \Og_K$ be a smooth, separated model of $P$ such that $$t(\mathscr{P}'_\kappa)=\delta(\mathscr{P}').$$ Then $P$ has pseudo-semiabelian reduction. 
\end{proposition}
\begin{proof}
We may assume without loss of generality that $\Og_K$ be strictly Henselian (\cite{BLR}, Chapter 7.2, Corollary 2). Let $\mathscr{P}\to \Spec \Og_K$ be the Néron model of $P.$ Let $\alpha:=\alpha(\mathscr{P}_ \kappa),$ $\alpha':=\alpha(\mathscr{P}'_\kappa),$ $t:=t(\mathscr{P}_\kappa),$ and $t':=t(\mathscr{P}'_\kappa).$ We shall first show that $t'\leq t.$ Let $T'$ and $T$ be the maximal tori in the special fibres of $\mathscr{P}'$ and $\mathscr{P}$, respectively. Let $\ell$ be a prime number invertible in $\Og_K$ which does not divide the number of irreducible components of the special fibre of either $\mathscr{P}$ or $\mathscr{P}'.$ If the induced morphism $T'\to T$ were not finite, the map $T'[\ell](\kappa)\to T[\ell](\kappa)$ would not be injective. However, the map $\mathscr{P}'[\ell](\Og_K)\to \mathscr{P}[\ell](\Og_K)$ is injective because $\mathscr{P}[\ell]$ and $\mathscr{P}'[\ell]$ are both separated and étale over $\Og_K$ and the map is injective generically. Further observe that 
$\ord \mathscr{P}[\ell](\Og_K)=\ell^{t+2\alpha},$ and similarly for $\mathscr{P}'[\ell](\Og_K).$ This shows that $t'+2\alpha'\leq t+2\alpha.$ Moreover, Proposition \ref{deltajumpprop} tells us that $t+\alpha\leq \alpha(P)=t'+\alpha'.$ Putting all these inequalities together, we obtain 
$$0\leq t-t'\leq \alpha'-\alpha\leq 2(\alpha'-\alpha)\leq t-t'.$$ This forces $\alpha=\alpha'$ and hence $t=t'.$ 
\end{proof}\\
$\mathbf{Remark.}$ The proof of the Proposition above also shows that the invariants $\alpha'$ and $t'$ do not depend on the choice of the particular model $\mathscr{P}'.$ \\
\\
We can now state and prove an analogue of Grothendieck's representation-theoretic criterion for semiabelian reduction of Abelian varieties:
\begin{theorem}
Let $\Og_K$ be an excellent discrete valuation ring with field of fractions $K.$ Let $K\sep$ be a separable closure and $I\subseteq \Gal(K\sep/K)$ an inertia subgroup. Let $P$ be a pseudo-Abelian variety over $K.$ Let $\ell$ be a prime number invertible in $\Og_K.$ Then the following are equivalent: \\
(i) The pseudo-Abelian variety $P$ has pseudo-semiabelian reduction over $\Og_K,$\\
(ii) The action of $I$ on $T_{\ell}(P)$ is unipotent, and \\
(iii) For all $\sigma\in I,$ we have $(\sigma-1)^2=0$ as operators on $T_\ell(P).$ \label{unipotentthm} 
\end{theorem}
\begin{proof}
(i) $\Rightarrow$ (iii): By \cite{T}, Theorem 2.1, we can find an Abelian variety $E$ and a unipotent group $V$ over $K$ together with an exact sequence $0\to E\to P\to V\to 0.$ Then the map $T_\ell(E)\to T_\ell(P)$ is an isomorphism, and it suffices to show that $E$ has semiabelian reduction (see \cite{Con2}, Theorem 5.5 and the remark thereafter). Let $\mathscr{E}$ and $\mathscr{P}$ be the Néron models of $E$ and $P$, respectively. Then the morphism $$\mathscr{E}[\ell']\to \mathscr{P}[\ell']$$ is an isomorphism for all prime numbers $\ell'$ invertible on $\Og_K$ (this follows from the Néron mapping property). Now Propositions \ref{Eprop} and \ref{alphainvariantprop} imply the claim. \\
(iii) $\Rightarrow$ (ii) is trivial.\\
(ii) $\Rightarrow$ (i): Since $T_\ell(E)\to T_\ell(P)$ is an isomorphism, we know that $E$ has semiabelian reduction over $\Og_K.$ Now the same argument as above shows that $P$ has pseudo-semiabelian reduction. 
\end{proof}
\begin{corollary}
Let $P$ be a pseudo-Abelian variety over $K.$ Let $L$ be a finite purely inseparable extension of $K$ over which $P$ is an extension of an Abelian variety $A$ by a smooth connected unipotent group $U.$ Further, we write $P$ as an extension of a smooth connected unipotent group $V$ by an Abelian variety $E$ over $K.$ Then the following are equivalent: \\
(i) $P$ has pseudo-semiabelian reduction over $\Og_K$,\\
(ii) $E$ has semiabelian reduction over $\Og_K,$ and \\
(iii) $A$ has semiabelian reduction over $\Og_L.$  \label{equivalentIIcor}
\end{corollary}
\begin{proof}
This follows from the previous Theorem together with the fact that the Galois representations $T_\ell(P),$ $T_\ell(A)$, and $T_\ell(E)$ are canonically isomorphic for all prime numbers $\ell\in \Og_K^\times.$
\end{proof}\\
Just as in the case of Abelian varieties, we have 
\begin{proposition}
Let $0\to P_1\to P_2\to P_3\to 0$ be an exact sequence of group schemes over $K$ all of whose elements are pseudo-Abelian varieties. Then $P_2$ has pseudo-semiabelian reduction if and only if so do $P_1$ and $P_3.$ 
\end{proposition}
\begin{proof}
We have an exact sequence $0\to T_\ell(P_1)\to T_\ell(P_2)\to T_\ell(P_3)\to 0$ for some choice of prime number $\ell\in \Og_K^\times.$ Now the Proposition follows from Theorem \ref{unipotentthm}.
\end{proof}
\begin{proposition}
Let $P$ be a pseudo-Abelian variety over $K$ and let $F/K$ be a finite separable extension. Let $S$ be the localization of the integral closure of $\Og_K$ in $F$ at a non-zero prime ideal. Then $P_F$ has pseudo-semiabelian reduction over $S.$ 
\end{proposition}
\begin{proof}
Write $P$ as an extension $0\to E\to P\to V\to 0$ of an Abelian variety $E$ and a smooth, connected, commutative unipotent group $V$ over $K.$ Since $P$ has pseudo-semiabelian reduction, it follows that $E$ has semiabelian reduction by Corollary \ref{equivalentIIcor}. Since $E_F$ has semiabelian reduction over $S,$ the same Corollary implies that $P_F$ has pseudo-semiabelian reduction, as desired.  
\end{proof}\\
$\mathbf{Remark}.$ If $P$ is a pseudo-Abelian variety over $K$ with pseudo-semiabelian reduction and Néron model $\mathscr{P}\to \Spec \Og_K,$ it does not follow that the unipotent radical of $\mathscr{P}^0_k$ is trivial. Indeed, suppose that the residue field $\kappa$ be algebraically closed. Then, if $P$ arises as the Weil restriction of an Abelian variety over a non-trivial purely inseparable extension of $K$ with good reduction, the special fibre of $\mathscr{P}$ is the Weil restriction of an Abelian scheme over a finite non-étale (hence non-reduced) $\kappa$-algebra $A.$ Using that the special fibre of $\mathscr{P}$ is not proper, the following Lemma will show that the unipotent radical of $\mathscr{P}^0_{\kappa}$ is non-trivial. Note that, in particular, the unipotent radical of the special fibre of the Néron model cannot be removed after any finite separable base change.  
\begin{lemma}
Let $\kappa$ be a field and let $B$ be a finite $\kappa$-algebra. Let $E$ be an Abelian scheme over $B.$ Then the maximal torus of $\Res_{B/\kappa}E$ is trivial. 
\end{lemma} 
\begin{proof}
First note that $\Res_{B/\kappa}E$ is representable by Proposition \ref{finitealgprop}.  We may assume that $\kappa$ be algebraically closed and show that there are no no-trivial maps $\Gm\to \Res_{B/\kappa}E,$ which is the same as showing that all maps $\Gm\to E$ over $B$ are trivial. This is clearly the case at the special point of $\Spec B,$ so our assertion follows from \cite{SGA3}, Exposé IX, Corollaire 3.5.
\end{proof}\\
In fact, this is a special case of a more general phenomenon: 
\begin{proposition}
Let $P$ be a pseudo-Abelian variety over $K$ and let $\mathscr{P}'$ be a smooth separated model of $P$ such that $\mathscr{P}'^0_\kappa$ is semiabelian (i.e., is an extension of an Abelian variety by a torus). Then $P$ is an Abelian variety with semiabelian reduction. If $\mathscr{P}'^0_\kappa$ is an Abelian variety then $P$ has good reduction.
\end{proposition}
\begin{proof}
Assume that $\Og_K$ be strictly Henselian. We begin by showing that the special fibre of the Néron model $\mathscr{P}\to \Spec \Og_K$ of $P$ must be semiabelian as well. Consider the canonical map $\mathscr{P}'\to \mathscr{P}.$ Then, for all prime numbers $\ell\in \Og_K^\times,$ the map $\mathscr{P}'[\ell](\Og_K)\to \mathscr{P}[\ell](\Og_K)$ is injective (because the map $\mathscr{P}'[\ell]\to \mathscr{P}[\ell]$ is an isomorphism at the generic fibre). For all such $\ell$, the map 
$\mathscr{P}'^0_{\kappa}[\ell](\kappa)\to \mathscr{P}^0_{\kappa}[\ell](\kappa)$ is injective, which implies that the morphism $\mathscr{P}'^0_{\kappa}\to \mathscr{P}^0_{\kappa}$ has finite kernel. Since source and target of this map have the same dimension, this implies that $\mathscr{P}^0_{\kappa}$ is semiabelian. Now let $E$ be the maximal Abelian subvariety of $P,$ and denote its Néron model by $\mathscr{E}.$ Then the Néron mapping property implies that $\mathscr{E}[\ell]\to \mathscr{P}[\ell]$ is an isomorphism, so $\mathscr{E}_{\kappa}$ and $\mathscr{P}_{\kappa}$ have the same toric and virtual Abelian ranks by Proposition \ref{alphainvariantprop}. This shows that $\dim_KE\geq \dim_KP,$ so $E\to P$ is an isomorphism. 
The remaining claims are now obvious.
\end{proof}\\
In the situation above, it does not suffice to show that the unipotent radical of $\mathscr{P}'_{\kappa}$ is trivial (even if $\mathscr{P}'=\mathscr{P}$), unless $\kappa$ is perfect. Indeed, suppose $\kappa$ is not perfect and let $a\in \Og_K$ be an element whose image in $\kappa$ does not have a $p$-th root for $p=\mathrm{char}\, \kappa.$ Let $L:=K[X]/\langle X^p-a\rangle.$ Then the integral closure $\Og_L$ of $\Og_K$ in $L$ is a discrete valuation ring such that the induced extension $\kappa\subseteq \kappa'$ of residue fields is purely inseparable of degree equal to $[L:K].$ In particular, $\kappa'=\Og_L\otimes_{\Og_K}\kappa.$ If $\mathscr{A}\to \Spec \Og_L$ denotes a semiabelian scheme with proper generic fibre $A$, then $P:=\Res_{L/K}A$ is a non-proper pseudo-Abelian variety over $K$ with Néron model $\mathscr{P}$ such that $\mathscr{P}^0=\Res_{\Og_L/\Og_K}\mathscr{A}.$ The special fibre of this scheme has no unipotent subgroups over $\kappa$ but is not semiabelian. 

\section{Étale cohomology of pseudo-Abelian varieties}
We shall keep the notation from the last paragraph; in particular, $\Og_K$ is an excellent discrete valuation ring with field of fractions $K.$ In this section, we shall show that, for a pseudo-Abelian variety $P$ over the field $K,$ the $\Gal(K\sep/K)$-representations $T_\ell(P)$ and $H^1(P_{K\sep}, \Z_\ell)$ are canonically dual to each other (just as in the case of Abelian varieties). Hence the representation-theoretic criteria for good reduction and pseudo-semiabelian reduction we proved above can be stated in terms of $H^1(P_{K\sep}, \Z_\ell)$ instead of Tate modules. 

\begin{lemma}
Let $P$ be a pseudo-Abelian variety over the field $K.$ Let $L$ be a finite, purely inseparable extension of $K$ over which $P$ is an extension of an Abelian variety $A$ by a smooth connected unipotent algebraic group $U.$ Let $\ell$ be a prime number invertible in $\Og_K.$ Then the induced morphism $H^1(A_{K\alg}, \Z_\ell)\to H^1(P_{K\alg}, \Z_\ell)$ is an isomorphism. 
\end{lemma}
\begin{proof}
First observe that $\Gamma(P_{K\alg}, \Og_{P_{K\alg}}^\times)=K^{\mathrm{alg},\times}.$ Indeed, let $f\colon P_{K\alg}\to \Gm$ be an element of the first group. After translating $f$ by a $K\alg$-point of $\Gm,$ we may suppose that $f(e)=1$ where $e$ is the unit element of $P(K\alg).$ Hence we may assume that $f$ is actually a homomorphism of algebraic groups by a theorem of Rosenlicht (see \cite{Ros}, Theorem 3). In particular, the restriction of $f$ to $U_{K\alg}$ vanishes, so $f$ pulls back from a homomorphism $A_{K\alg}\to \Gm,$ which clearly vanishes as well. This shows that our original $f$ is constant. Now consider the exact sequence $0\to \boldsymbol{\mu}_{\ell^n}\to \Gm\to \Gm\to 0$ of étale sheaves on $P_{K\alg}.$ We obtain a commutative diagram with exact rows
$$\begin{CD}
0@>>> H^1(P_{K\alg}, \Z/\ell^n\Z)@>>> \Pic P_{K\alg}@>{\cdot \ell^n}>> \Pic P_{K\alg}\\
&&@AAA@AAA@AAA\\
0@>>> H^1(A_{K\alg}, \Z/\ell^n\Z)@>>> \Pic A_{K\alg}@>>{\cdot \ell^n}> \Pic A_{K\alg},
\end{CD}$$
which reduces our claim to showing that the morphism $\Pic A_{K\alg}\to \Pic P_{K\alg}$ is an isomorphism. The map $P_{K\alg}\to A_{K\alg}$ turns $P_{K\alg}$ into an $U_{K\alg}$-torsor over $A_{K\alg}.$ Since $U_{K\alg}$ has a composition series with successive quotients isomorphic to $\mathbf{G}_{\mathrm{a}}$ (see \cite{SGA3}, Exposé XVII, Proposition 4.1.1), we see that this torsor is trivial locally in the Zariski topology, and that $\Pic U_{K\alg}=0.$ Hence $\Pic A_{K\alg}\to \Pic P_{K\alg}$ is an isomorphism by \cite{FossI}, Proposition 3.1.
\end{proof}
\begin{proposition}
Let $P$ be a pseudo-Abelian variety over $K$ and let $\ell$ be a prime number invertible in $K.$ Then there is a $\Gal(K\sep/K)$-equivariant perfect pairing 
$$T_\ell(P)\times H^1(P_{K\sep}, \Z_\ell)\to \Z_\ell.$$
\end{proposition}
\begin{proof}
Using topological invariance of the étale site (see Proposition \ref{topolequivprop}), we may replace $K\sep$ by $K\alg$ in the statement, and write $P$ as an extension $0\to U_{K\alg}\to P_{K\alg}\to A_{K\alg}\to 0$ as above. Then the maps $T_\ell(P)\to T_\ell(A)$ and $H^1(A_{K\alg}, \Z_\ell)\to H^1(P_{K\alg}, \Z_\ell)$ are isomorphisms which are clearly $\Gal(K\sep/K)$-equivariant. Hence the existence of our desired pairing follows because a Galois equivariant perfect pairing 
$$T_\ell(A)\times H^1(A_{K\alg}, \Z_\ell)\to \Z_\ell$$ is well-known to exist. 
\end{proof}
\begin{corollary}
In Theorems \ref{NOScritthhm} and \ref{unipotentthm}, we may replace $T_\ell(P)$ by $H^1(P_{K\sep}, \Z_\ell)$ without affecting the validity of those criteria. The same is true if we replace $T_\ell(P)$ by $T_\ell(P)\otimes_{\Z_\ell}\Q_\ell$ of $H^1(P_{K\sep}, \Q_\ell),$ since $T_\ell(P)$ and $H^1(P_{K\sep}, \Z_\ell)$ are torsion-free.
\end{corollary}
$\mathbf{Acknowledgement}.$ The author would like to express his gratitude to Dr. G. Gagliardi and Dr. D. Gvirtz for helpful discussions.

\end{document}